\documentclass[12pt]{article}
\usepackage{amsmath}
\usepackage{amssymb}
\usepackage{amsthm}
\usepackage[margin=2.5cm]{geometry}
\usepackage{bbm}
\usepackage{amsthm}
\usepackage{mathtools}
\usepackage[mathscr]{euscript}
\usepackage[utf8]{inputenc}
\usepackage{color}
\usepackage{enumerate} 
\usepackage[shortlabels]{enumitem} 
\usepackage{bm}
\usepackage{graphicx}
\usepackage[hidelinks]{hyperref}
\usepackage[font=small,labelfont=bf]{caption}

\newcommand{\R}[0]{\mathbb{R}}

\newcommand{\Z}[0]{\mathbb{Z}}
\newcommand{\N}[0]{\mathbb{N}}

\renewcommand{\P}[0]{\mathbb{P}}
\newcommand{\E}[0]{\mathbb{E}}

\newcommand{\sB}[0]{\mathcal{B}}

\newcommand{\sM}[0]{\mathcal{M}}

\newcommand{\sO}[0]{\mathcal{O}}
\newcommand{\sP}[0]{\mathcal{P}}

\newcommand{\sU}[0]{\mathcal{U}}
\newcommand{\sV}[0]{\mathcal{V}}

\newcommand{\sL}[0]{\mathcal{L}}
\newcommand{\sI}[0]{\mathcal{I}}

\newcommand{\sX}[0]{\mathcal{X}}

\newcommand{\il}[0]{\langle}
\newcommand{\ir}[0]{\rangle}

\newcommand{\bj}[0]{{\bf j}}

\newcommand{\bx}[0]{{\bf x}}

\newcommand{\eps}{\varepsilon}

\newtheorem{theorem}{Theorem}[section]
\newtheorem{corollary}[theorem]{Corollary}
\newtheorem{lemma}[theorem]{Lemma}
\newtheorem{prop}[theorem]{Proposition}
\newtheorem{definition}[theorem]{Definition}
\newtheorem{remark}[theorem]{Remark}
\newtheorem{condition}[theorem]{Condition}

\numberwithin{equation}{section}

\allowdisplaybreaks

\title{Point process convergence of extremes in $K$-symmetric exclusion}
\author{Michael Conroy\thanks{Corresponding author, Clemson University, \texttt{meconro@clemson.edu}}, Adri\'an Gonz\'alez Casanova\thanks{Arizona State University, \texttt{agonz591@asu.edu}}, and Sunder Sethuraman\thanks{University of Arizona, \texttt{sethuram@arizona.edu}}}
\begin{document}
\maketitle

\begin{abstract}
We consider the behavior of extremal particles in $K$-symmetric exclusion on $\Z$ when the process starts from certain infinite-particle step configurations where there are no particles to the right of a maximal one.  In such a system, the occupancy of a site is limited to at most $K\geq 1$.  Let $X^{(0)}_t\geq X^{(1)}_t\geq \cdots$ denote the order statistics of the particles in the system. We show that the point process $\sum_{m=0}^\infty \delta_{v_t(X_{t/K}^{(m)})}$ converges in distribution as $t \to \infty$ to a Poisson random measure on $\R$ with intensity proportional to $e^{-x}\,dx$, where $v_t(x) = (\sigma b_t)^{-1}x - a_t$, $a_t = \log(t/ (\sqrt{2\pi} \log t))$, $b_t = (t/\log t)^{1/2}$, and $\sigma$ is the standard deviation of the random walk jump probabilities. 

From this limit, we further deduce the asymptotic joint distributions for the extreme statistics and the spacings between them.
Moreover, to probe effects of the number of particles on the behavior of the extremes, we consider an array of truncated step profiles supported on blocks of $L(t)$ sites at times $t\geq 0$. Letting $L(t) \to \infty$ with $t \to \infty$, we obtain Poisson random measure limits in different scaling regimes determined by $L(t)$. 

These results show robustness of both previously known and newly introduced superdiffusive scaling limits for the extremes in the symmetric exclusion process ($K=1$) by extending them to the larger class of $K\geq 2$ exclusion. Furthermore, proofs are more general than previously known techniques, relying on moment bounds and a semigroup monotonicity estimate to control particle correlations. 

\end{abstract}

\section{Introduction}

Recently, for the symmetric exclusion processes (SEP) on $\Z$ with nearest-neighbor interactions,
the maximum position $X^{(0)}_t$ of particles starting at the negative integers (among other more general starting configurations) was shown to satisfy a Gumbel limit
\begin{align}
\label{first limit}
b_t^{-1}X^{(0)}_t - a_t \underset{t\to\infty}\Longrightarrow G(x) = e^{-e^{-x}}, 
\end{align}
for $a_t = \log(t/(\sqrt{2\pi}\log t))$ and $b_t = (t/\log t)^{1/2}$; see
\cite{ConSet2023}, and also \cite{ConSet2025} for generalizations to symmetric exclusion on $\Z^d$.
Such a phenomenon is not captured by the `hydrodynamics' of the process:  In diffusive scale, at positive macroscopic times, the bulk mass will have infinite support on $\R$, the lead particle having already gone to `infinity'.  However, given a maximum occupancy per site, and despite the symmetry of interactions, the lead particle is driven to the right by the mass of particles behind it in a superdiffusive scale $a_tb_t \sim \sqrt{t\log t}$.

Interestingly, the limit \eqref{first limit} holds for the lead particle when the particles do not interact \cite{Arr1983}.
In this case of independent particle motion, the superdiffusive displacement of the maximum is due to the multiplicity of options for the argmax. In contrast, the rigidity of nearest-neighbor SEP means that the maximal particle is identified at $t = 0$, so that the superdiffusivity of its displacement is an effect of a `push' from the bulk mass of particles. It turns out that the extremal particle limits
resulting from these two phenomena match exactly. 

A motivating question is whether a form of the Gumbel limit is universal for the extremes in particle systems with symmetric interactions, highly nonequilibrium initial conditions, and finite range or `light tailed' jump distribution.  
In this paper, we test this hypothesis by examining the behavior of extreme particles starting from types of step profiles in a generalization of SEP, namely the
$K$-symmetric exclusion process ($K$-SEP) on $\Z$, where $K$ particles are allowed per site. 
Informally, $K$-SEP follows a collection of continuous time random walks such that a particle at $x$ jumps to $y$ with symmetric rate proportional to $p(x,y)$, when the number of particles at $y$ is strictly less than $K$, and is suppressed otherwise.    Although $K$-SEP---also known as `partial exclusion' or `generalized exclusion'---is a natural extention of exclusion interaction, 
it is significantly less studied than SEP. 
For known results, see \cite{Cocozza} where it seen as a type of `misanthrope process',  \cite{Kiesling} (and \cite{Arvind} for an asymmetric version) for invariant measures, \cite{FRS,FGS,KipLanOll1994} (and \cite{Seppalainen} for an asymmetric version) for hydrodynamics, \cite{Caputo,eyob} for mixing times, \cite{GHS} for gradient condition and fluctuations of the density fields, \cite{Xue} for occupation times, and references therein.

 $K$-SEP for $K \ge 2$ is a less rigid model than SEP in that, even when $p$ is nearest-neighbor, the particles are not ordered over time. Since larger values of $K$ correspond, in a sense, to less rigidity, 
one may view $K$-SEP as intermediate between SEP ($K=1$) and independent particles ($K=\infty$). In fact, it was shown in \cite[Theorem 4.3]{GiaKurRedVaf2009} that when time is rescaled by $K^{-1}$ and then $K \to \infty$, the $K$-SEP system at a fixed time converges to a system of non-interacting symmetric random walks. Nevertheless, $K$-SEP is a process without several of the simplifying features of SEP or independent particles, and so in this sense is a natural model in which to study the phenomenon. It is also of interest to understand the dependence on $K$ in the extreme particle limits.

\subsection{Discussion of results}

Our main result for $K$-SEP (Theorem \ref{main})  establishes a Poisson random measure limit for the point process of order statistics of particle positions, 
$N_t = \sum_{m=0}^\infty \delta_{X_t^{(m)}}$,
when starting from types of step profiles. Namely, when initially the system has $K$ particles at every site to the left of $0$ and none to the right, 
 \begin{align}
 \label{Poisson first limit}
 N_{t/K} \circ v_t^{-1} = \sum_{m=0}^\infty \delta_{v_t(X^{(m)}_{t/K})} \underset{t\to\infty}{\Longrightarrow} \text{PRM}\,(K\sigma \lambda), 
 \end{align} 
 where $v_t(x) = (\sigma b_t)^{-1}x - a_t$, $\sigma$ is the standard deviation of $p$, `$\Rightarrow$' indicates convergence in distribution, and 
PRM$(K\sigma \lambda)$ denotes the law of a Poisson random measure on $\R$ with intensity proportional to $\lambda(dx) = e^{-x}\,dx$.

 The time rescaling by $K^{-1}$ in \eqref{Poisson first limit} standardizes the jumps of the underlying random walks in the $K$-SEP system to occur at rate $1$ (and hence no $K$ appears in the formula for $v_t$). Notably, a factor of $K$ still appears in the limiting intensity. This is due to the initial condition, namely that $K$ `layers' of particles are placed at the nonpositive integers (see Figure \ref{fig:ladder} below). As a consequence of Theorem \ref{Lmain} (a), if instead the initial system has $j \le K$ particles at every site to the left of $0$ and none to the right, then the right hand side of \eqref{Poisson first limit} becomes PRM$\,(j\sigma \lambda)$. In particular, we may achieve a limit in \eqref{Poisson first limit} for any $K$ by choosing $j=1$, an initial condition allowable in all $K$-SEP systems. Theorem 3.5 (a) also presents the corresponding results for a random step initial condition determined by a product measure with a certain amount of spatial uniformity or periodicity (Condition \ref{initialprofilecond}).

The result in \eqref{Poisson first limit} is new even for SEP and it considerably extends \eqref{first limit}.  For instance, we consequently derive scaling limits for the joint distributions of order statistics and for the spacings between extremes (Corollaries \ref{orderstatlimit} and \ref{spacings}).  We also remark here that the Poisson process limit for SEP, with respect to the mapping to the symmetric `corner growth' model (cf. \cite{LigBook05}), corresponds to the scaling limit of the bottom horizontal and left vertical edges.
While Poisson process limits are known classically for independent random variables \cite{Res}, for systems of independent random walks \cite{MikYsl2020}, and recently for particle systems with mean-field interactions \cite{kolliopoulos2023,kolliopoulos2023point}, 
 the Poisson process limit \eqref{Poisson first limit} shown here appears to be the first in the context of symmetric particle systems with local interactions. 
 
In our second main result (Theorem \ref{Lmain}), we determine the extent of influence of particles starting behind the origin on the extremal particle and Poisson limit behaviors in $K$-SEP.  We consider step profiles supported on blocks of size $L(t)$ where $\lim_{t\to\infty} L(t) = \infty$, inducing a sequence of $K$-SEP processes in which the number of particles in the system becomes infinite as $t \to \infty$. Consequently, we obtain three limiting regimes. If $L(t)/b_t \to \infty$, the same Poisson point process limit as in \eqref{Poisson first limit} when the step profile consisted of an infinite number of particles is obtained. If $L(t)/b_t \to \psi \in (0,\infty)$, then a variation of \eqref{Poisson first limit} holds with an additional proportionality constant $C(\psi) < 1$ for the limit intensity. 
In constrast, when $L(t)/b_t \to 0$ the extreme particle behavior is on a different scale of order $\sqrt{t\log L(t)}$, and in this scale a corresponding Poisson point process limit holds, also with intensity proportional to $\lambda$.  One may understand from these results that particles starting 
at distances orders of magnitude greater than $b_t$ play no role in the limiting Poisson point process, likely because they lack sufficient time to surpass the lead particle,
while those starting closer to the origin can contribute to the scale of the maximal particle.  These Poisson process limits improve upon related `$L$-step' Gumbel limits analogous to \eqref{first limit} shown for SEP in \cite{ConSet2023}.

\subsection{Proof techniques}

The analysis of $K$-SEP is more challenging than for SEP, where many more properties are known.
Consequently, our proof techniques are notably different than those in \cite{ConSet2023}. The analysis of SEP is aided by 
the `strong Rayleigh' property, 
a type of negative association of Bernoulli random variables which, in the context of SEP, implies that a sum of occupation variables at a fixed time $t$ has the same distribution as a sum of independent Bernoulli variables \cite{BorBraLig2009,Lig2009,Van2010}. 
While the measure $N_t$ may be expressed in terms of the occupation variables, in $K$-SEP they are $\{0, 1, \ldots, K\}$-valued, and
nothing analogous to the strong Rayleigh property is known.  
One of our main motivations is the view that developing techniques to handle $K$-SEP interactions would help understand better the robustness of limits such as \eqref{first limit} in other symmetric particle systems.

To show Poisson point process convergence,   
we describe $K$-SEP evolution in terms of `stirring' variables and use 
the moment method, specifically identifying the limits of expectations of the factorial measure associated with $N_t$.  
To control particle correlations, in Section \ref{semigroup} we develop general estimates that compare sums involving the $K$-SEP semigroup to sums involving the semigroup of independent motion that could be useful in other contexts. 
A key ingredient is a monotonicity estimate in \cite{GiaRedVaf2010} which bounds the $K$-SEP semigroup by that of a system of independent particles (Lemma \ref{VUineq}). As we show in Proposition \ref{factmomentbound}, the $K$-SEP and independent particle comparison arises naturally in the analysis of the factorial measure of $N_t$. 
To finish the proof, we apply asymptotic results for a single continuous time random walk, which are based on sharp tail bounds and a local central limit theorem.    

In the context of robustness of \eqref{first limit} and \eqref{Poisson first limit} to other types of symmetric interaction, this argument suggests that sufficient ingredients are a representation in terms of component random walks (e.g., the system of stirring variables)
and a form of 
semigroup monotonicity that allows the comparison of these components to their uncorrelated counterparts. 
Seemingly, this is a less stringent requirement than the strong Rayleigh property. 
Such conditions are satisfied for the variant of $K$-SEP considered here, in which a particle jumps from site $x$ to site $y$ with rate 
$p(x,y)\eta(x)(K- \eta(y))$, 
where $\eta(x)$ and $\eta(y)$ denote the number of particles at $x$ and $y$, respectively (the complete generator of the process is given below in \eqref{eq:KSEPgen}).

\section{Preliminaries}

Here, we define the $K$-exclusion process (Section \ref{processdef}), specify the initial conditions of interest (Section \ref{stepprofiles}), formulate the pre-limit point process (Section \ref{pointprocessdef}), and give criteria for weak convergence to a Poisson process (Section \ref{weakconvPRM}). 

\subsection{$K$-SEP and the stirring process}\label{processdef}

For fixed $K \in \{1, 2, \ldots\}$, we consider the $K$-particle-per-site symmetric exclusion process ($K$-SEP) on $\Z$ with jump distribution $p(\cdot, \cdot)$. Here, $p(x,y) = p(y,x) = p(0,y-x)$ are symmetric, irreducible, and translation invariant random walk transition probabilities on $\Z$ with $p(0,0) = 0$.  We allow $p(0, \cdot)$ to be infinite-range, as long as it has finite moment generating function near $0$. The assumptions on $p$ are stated more formally in Section \ref{results}.

Throughout, we denote $[K] = \{1, 2, \ldots, K\}$ and $[K]_0 = [K] \cup\{0\}$ for a positive integer $K$. $K$-SEP may be defined as the Markov process $\{\eta_t :t \ge 0\}$ on the state space $\sX_K := [K]_0^\Z$ with generator (defined for any $f$ depending on $\eta$ through only finitely-many marginals $\eta(x)$) given by
\begin{equation}\label{eq:KSEPgen}
	\sL_K f(\eta) = \sum_{x,y\in\Z} p(x,y)\eta(x)(K - \eta(y))(f(\eta^{x,y}) - f(\eta)), 
\end{equation}
where $\eta^{x,y}(x) = \eta(x) - 1$, $\eta^{x,y}(y) = \eta(y) + 1$, and $\eta^{x,y}(u) = \eta(u)$ for $u \ne x,y$. For each $t \ge 0$ and $x \in \Z$, the `occupation variable' $\eta_t(x)$ gives the number of particles that sit at $x$ at time $t$. At the particle level, the dynamics of the process $\eta_t$ are described as follows. Each particle in the system attempts to move as a continuous time random walk with jumps at rate $K$ and transition probabilities $p$, but jumps to sites occupied by $K$ particles are suppressed. When $K = 1$, we get the standard SEP dynamics where jumps are only allowed to unoccupied sites in $\Z$. $K$-SEP may also be viewed as $1$-SEP on the so-called  `ladder graph' with vertices $\Z \times [K]$ and edges $\{\{(x,j), (y,k)\} : p(x,y) > 0\}$ depicted in Figure \ref{fig:ladder}.

There is another construction of the process $\eta_t$ that is particularly fruitful for our analysis, which is based on a system of correlated random walks called `stirring' variables. Place $K$ particles initially at every point in $\Z$, given the labels $1, 2, \ldots, K$. To each pair $\{x, y\} \subset \Z$, associate a Poisson process $\theta_{xy}$ with rate $K^2p(x,y)$, where $\{\theta_{xy} : \{x, y \} \subset \Z\}$ are independent. At each event time of $\theta_{xy}$, a particle at $x$ and a particle at $y$ are selected uniformly at random and interchanged. The stirring system $\{\xi_t^{xj} : (x,j) \in \Z \times [K]\}$ is defined by 
\[
	\xi_t^{xj} = \text{the position at time $t$ of the particle initially at $x$ with label $j$}. 
\]
For $\eta \in \sX_K$, the $K$-SEP process $\eta_t$ with $\eta_0 = \eta$ may then be specified by
\begin{equation}\label{eq:etadefasstir}
	\eta_t(x) = \sum_{y \in \Z} \sum_{j=1}^K 1(\eta(y) \ge j, \xi_t^{yj} = x), \qquad x \in \Z, 
\end{equation}
since the right hand side counts the number of stirring variables with position $x$ at time $t$. 

\begin{figure}[t] 
\captionsetup{width=.93\linewidth}
\centering
\includegraphics[scale=0.22]{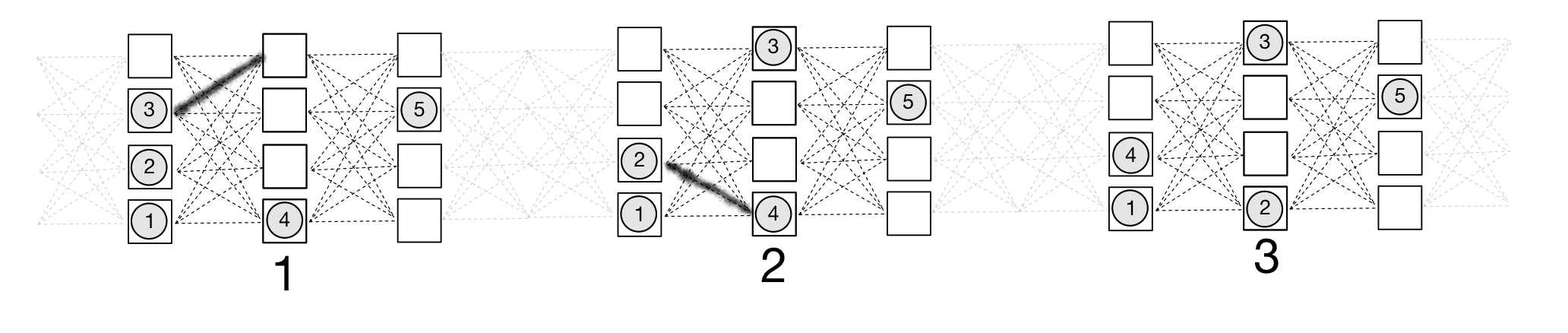}
\caption{We observe two consecutive particle movements over time on the nearest-neighbor ladder graph with $K = 4$. In the first step, a particle shifts one position to the right and lands in an available spot. In the second step, particle 2  attempts to move to the position occupied by 4, but the exclusion rule prevents its motion, as particles cannot move into already occupied sites.
In contrast, under the stirring coupling, we observe a different behavior: particles 2 and 4 swap positions, as shown in the third panel of the figure.}\label{fig:ladder}
\end{figure}

Figure \ref{fig:ladder} depicts the similarities and differences on the ladder graph between nearest-neighbor SEP, where movement into occupied sites is forbidden, and the stirring coupling, where movement into occupied sites leads to swaps. While in SEP the rate at which a particle moves depends on the local configuration of occupied and unoccupied neighboring sites, an important advantage of the stirring coupling is that each particle moves at a constant rate, regardless of the occupancy of neighboring positions.

\eqref{eq:etadefasstir} gives a convenient description of $\eta_t$ that is the basis for 
much of our analysis. 
Thus we assume $\{\xi_t^{xj} : (x, j) \in \Z \times [K], t \ge 0\}$ and $\{\eta_t : t \ge 0\}$ are coupled according to \eqref{eq:etadefasstir} on the same space with probability measure $\P$ and corresponding expectation operator $\E$. Then for $\eta \in \sX_K$, we define $\P_\eta(\cdot) = \P(\cdot | \eta_0 = \eta)$ with corresponding expectation $\E_\eta$. 
For a probability measure $\nu$ on $\sX_K$, $E_\nu$ will denote expectation on $\sX_K$ with respect to $\nu$, and $\eta$ will be used to denote the canonical coordinate variable on $\sX_K$, i.e., $E_\nu[f(\eta)] = \int_{\sX_K} f(\eta)\,\nu(d\eta)$. 
Then $\P_\nu$, the measure under which $\eta_0 \sim \nu$, is defined by $\P_\nu(\cdot) = E_\nu[\P_\eta(\cdot)]$. Correspondingly, $\E_\nu[\cdot] = E_\nu[\E_\eta[\cdot]]$. 

We note that, from the above construction of the stirring system, the marginal distribution of a single trajectory $\{\xi_t^{xj} : t \ge 0\}$ is that of a continuous time random walk with jumps at rate $K$, transition kernel $p$, and starting point $x$. Throughout, $\{\zeta_t : t \ge 0\}$ will denote a continuous time random walk on $\Z$ with jumps at rate $1$ and transition probabilities $p$, and for each $x \in \Z$, $P_x$ will denote the probability measure under which $\zeta_0 = x$. Then, 
\begin{equation}\label{eq:stirmarg}
	\P(\xi_t^{xj} \in \cdot) = P_x(\zeta_{Kt} \in \cdot), 
\end{equation}
for any $x \in \Z$, $j \in [K]$, and $t \ge 0$.

\subsection{Step profiles}\label{stepprofiles}

For $\eta \in \sX_K$ and $A \subset [K]$, we introduce the shorthand $\{\eta \in A\} = \{x \in \Z : \eta(x) \in A\}$. For a probability measure $\nu$ on $\sX_K$,
we (loosely) refer to the set $\{x \in \Z : \nu(\eta(x) > 0) > 0\}$ as the support of $\nu$, and use the shorthand $\{\nu > 0\}$. 

The initial conditions of interest are `step' profiles, namely elements of 
\begin{equation}\label{eq:stepprofiledef}
	\sX_K^{\text{step}} := \{\eta \in \sX_K : \eta(x) = 0 \;\,\text{for all}\;\, x > 0\}. 
\end{equation}
Throughout, for $L \in \N \cup \{\infty\} = \{0, 1, 2, \ldots, \infty\}$, we let $\eta^L \in \sX_K^{\text{step}}$ denote the deterministic `$L$-step' profile to which we will compare other initial conditions: 
\begin{equation}\label{eq:fullstep}
	\eta^L(x) = K 1 ( - L < x \le 0 ), \qquad x\in \Z. 
\end{equation}
Then, for example, $\eta^\infty$ denotes the infinite `full step' $\eta^\infty(x) = K1(x \le 0)$, which is the maximal element in $\sX_K^{\text{step}}$ in the sense that $\eta(x) \le \eta^\infty(x)$ for all $\eta \in \sX_K^{\text{step}}$ and $x \in \Z$. It will be useful to have the shorthand $\P_L = \P_{\eta^L}$ and $\E_L = \E_{\eta^L}$ for the probability measure and expectation operator conditional on $\eta_0 = \eta^L$. We take the convention that any sequence denoted $\{L\}$ will be $\N \cup\{\infty\}$-valued. 

More generally, we will require initial distributions to be of product type (which includes deterministic ones), and we introduce the following notation. 
Analogously to \eqref{eq:stepprofiledef}, let 
$\sP_K$ denote the set of product probability measures on $\sX_K$, and 
\[
	\sP_K^{\text{step}} = \{\nu \in \sP_K : \{\nu > 0\} \subset \{\ldots, -2, -1, 0\}\}. 
\]

In the `$L$-step' context, we consider the following truncated product measures. For $\nu \in \sP_K$ and $L \in \N \cup \{\infty\}$, let 
\begin{equation}\label{eq:Lmeasure}
	\nu_L(\eta(x) \in \cdot) = \begin{cases} \nu(\eta(x) \in \cdot) & -L < x \le 0, \\ \delta_0 & \text{otherwise.} \end{cases}
\end{equation}
In particular, note that $\nu_\infty = \nu$ when $\nu \in \sP_K^{\text{step}}$.

We are interested in the asymptotics of the point process of particles in the $K$-SEP system. The collection of particle positions may be equivalently described via the stirring variables defined above or by order statistics. Namely, under an initial condition where $\eta_0 \in \sX_K^{\text{step}}$ a.s., let  
\[
	\cdots \le X_t^{(3)} \le X_t^{(2)} \le X_t^{(1)} \le X_t^{(0)} < \infty
\]
denote the order statistics of the set $\{\eta_t > 0\}$. So, in particular, 
\[
	X_t^{(0)} = \max\{\eta_t > 0\}. 
\]
When there are a finite number $k$ particles in the system (e.g., for an initial profile $\eta^L$ for $L < \infty$), we take the convention that $X_t^{(m)} = -\infty$ for $m > k$, so that we may define our point process below in terms of all $\{X_t^{(m)} : m \in \N\}$.

\subsection{The point process of particles}\label{pointprocessdef}

We say that a locally finite measure $\mu$ on $(\R, \sB_\R)$ is a counting measure if $\mu(A) \in \N \cup \{\infty\}$ for all Borel sets $A \in \sB_\R$. Any counting measure can be expressed in the form $\mu = \sum_{i \in S} n_i\delta_{x_i}$, called its atomic decomposition, where $S$ is countable, $\{n_i\} \subset \N$, $\{x_i\} \subset \R$ are distinct, and $\delta_x$ denotes the Dirac point mass at $x$. 

We call a random measure $M$ on $(\R, \sB_\R)$ a point process when $M$ is a counting measure almost surely. In this case $M$ may be written in terms of its atomic decomposition as $M = \sum_{i \in S} \alpha_i \delta_{\beta_i}$, where $\{\beta_i\}$ are a.s. distinct random variables in $\R$ called points and $\{\alpha_i\}$ are integer-valued random weights called multiplicities. The point process $M$ is said to be simple if $\alpha_i = 1$ a.s. for all $i$.

Our main object of study is the sequence of point processes $\{N_t : t \ge 0\}$ on $\R$ of particle positions in the $K$-SEP system with step initial condition, which can be written in terms of the order statistics as 
\begin{equation}\label{eq:NdefOS}
	N_t = \sum_{m=0}^\infty \delta_{X_t^{(m)}}. 
\end{equation}
In words, $N_t(A)$ counts the number of particles in $A \in \sB_\R$ at time $t$ (note that $\eta_0 \in \sX_K^{\text{step}}$ ensures $X_t^{(0)} < \infty$ a.s.). Since at each time $t \ge 0$ every particle corresponds to exactly one stirring variable, we have the equivalent expression (cf. \eqref{eq:etadefasstir}) 
\begin{equation}\label{eq:NdefStirring}
	N_t = \sum_{x \in \Z} \sum_{j=1}^K 1(\eta_0(x) \ge j) \delta_{\xi_t^{xj}}. 
\end{equation}
We mainly use this latter form of $N_t$ in our analysis. However, the representation \eqref{eq:NdefOS} will allow us to derive limit distributions for the order statistics in Section \ref{OSlimits}.
While we do not use it explicitly, yet another form of $N_t$ is obtained through the occupation variables, namely $N_t(A) = \sum_{x \in A \cap \Z} \eta_t(x)$ for $A \in \sB_\R$. 

Since we wish to obtain a weak limit for $N_t$ as $t \to \infty$, the points of \eqref{eq:NdefOS} (equivalently, \eqref{eq:NdefStirring}) need to be rescaled. To get a sense of a scaling that achieves tightness, we may consider the mean measure or intensity of $N_t$:
From \eqref{eq:stirmarg} and \eqref{eq:NdefStirring}, when $\eta_0 \sim \nu$,  
\begin{equation}\label{eq:meanrep}
\begin{aligned}
	\mu_t^\nu(\cdot) := \E_\nu[N_t(\cdot)] &= \sum_{x \in \Z} \sum_{j=1}^K \nu(\eta(x) \ge j)P_x(\zeta_{Kt} \in \cdot) \\
	&= \sum_{x \in \Z} E_\nu[\eta(x)]P_x(\zeta_{Kt} \in \cdot).  
\end{aligned}
\end{equation}
(When $\nu$ is a point mass at $\eta$, we use the simplified notation $\mu_t^\eta = \mu_t^{\delta_\eta}$. Moreover, we write $\mu_t^L = \mu_t^{\eta^L}$ for $\eta^L$ in \eqref{eq:fullstep}.)
In particular, for a Borel set of the form $(y, \infty)$ and when $\nu \in \sP_K^{\text{step}}$, we see that 
\begin{align*}
	\mu_t^\nu(y, \infty) &= \sum_{x \le 0} E_\nu[\eta(x)] P_x(\zeta_{Kt} > y) \\
	&= \sum_{x \ge 0} E_\nu[\eta(-x)] P_0(\zeta_{Kt} - y > x) = E_0\bigg[ \sum_{x = 0}^{(\zeta_{Kt} - \lfloor y \rfloor)_+ - 1} E_\nu[\eta(-x)] \bigg], 
\end{align*}
where $c_+ = \max\{c, 0\}$ and $\lfloor c \rfloor = \max\{n \in \Z : n \le c\}$. 

For example, when $\nu$ is the point mass at $\eta^\infty$ (in \eqref{eq:fullstep}), 
\begin{equation}\label{eq:mu(a,infty)}
	\mu_t^\infty(y, \infty) = KE_0[(\zeta_{Kt} - \lfloor y \rfloor)_+]. 
\end{equation}
Scaling time by $K^{-1}$ and applying a scaling map of the form $v_t(x) = x/b_t - a_t$ to the points in \eqref{eq:NdefStirring} changes \eqref{eq:mu(a,infty)} to 
\[
	KE_0[(\zeta_t - v_t^{-1}(\lfloor y \rfloor))_+] = KE_0[(\zeta_t - b_t(\lfloor y \rfloor + a_t))_+]. 
\]
For the above display to converge as $t \to \infty$ to a finite value, $a_tb_t$ must be superdiffusive, namely $t^{-1/2}a_tb_t \to \infty$.

We specify the scaling sequence as follows. Define $\{(a_t, b_t) : t > 1\}$ by 
\[
	a_t = \log \bigg( \frac{t}{\sqrt{2\pi}\log t} \bigg), \qquad b_t = \bigg( \frac{t}{\log t} \bigg)^{1/2}, 
\]
and define the sequence of maps $v_t : \R \to \R$ by 
\[
	v_t(x) = \frac{x}{\sigma b_t} - a_t, \qquad x \in \R, 
\]
where $\sigma = (\sum_{x \in \Z} x^2 p(0,x))^{1/2}$, which we assume to be finite. Then to first order, 
\begin{equation}\label{eq:zfirstorder}
	v_t^{-1}(x) = \sigma b_t(x + a_t) \sim \sigma \sqrt{t\log t}, \qquad t \to \infty, 
\end{equation}
for all $x \in \R$, where $f(t) \sim g(t)$ as $t \to \infty$ means $\lim_{t\to\infty} f(t)/g(t) = 1$. 

Invertibility of $v_t$ is implicitly noted in \eqref{eq:zfirstorder}. 
We use the same notation for the sequence of preimage maps $v_t^{-1} : \sB_\R \to \sB_\R$, defined in the usual way by $v_t^{-1}(A) = \{x \in \R : v_t(x) \in A\}$. 
In particular, since $v_t$ is nondecreasing, 
\begin{equation}\label{eq:zinvinterval}
	v_t^{-1}(x, y] = (v_t^{-1}(x), v_t^{-1}(y)]. 
\end{equation}
We also note the property that for a collection of intervals $(x_k, y_k] \subset \R$, 
\begin{equation}\label{eq:zdisjoint}
	\bigcap_{k} (x_k, y_k] = \varnothing \quad\text{if and only if} \quad \bigcap_k v_t^{-1}(x_k, y_k] = \varnothing. 
\end{equation}

The main prelimit object of interest is the time-rescaled sequence of transformed point processes 
\begin{align*}
	N_{t/K} \circ v_t^{-1} = N_{t/K}(v_t^{-1}(\cdot)) &= \sum_{m=0}^\infty \delta_{v_t(X_{t/K}^{(m)})} = \sum_{x \in \Z} \sum_{j=1}^K 1(\eta_0(x) \ge j) \delta_{v_t(\xi_{t/K}^{xj})}, 
\end{align*}
with corresponding intensities 
\[
	\mu_{t/K}^\nu \circ v_t^{-1} = \sum_{x \in \Z} E_\nu[\eta(x)] P_x(v_t(\zeta_t) \in \cdot) = \sum_{x \in \Z} E_\nu[\eta(x)] P_x(\zeta_t \in v_t^{-1}( \cdot )). 
\]
We will use the composition notations $N_{t/K} \circ v_t^{-1}$ and $N_{t/K}(v_t^{-1}(\cdot))$ interchangeably throughout. 

For considering a truncated initial condition $\nu_L$ based on a sequence $\{L\} \subset \N$ with $L \uparrow \infty$ as $t\to\infty$, we introduce an alternate scaling sequence $(a_{t,L}, b_{t,L})$ given by 
\[
	a_{t,L} = \log \bigg( \frac{L^2}{\sqrt{2\pi \log L^2}} \bigg), \qquad b_{t,L} = \bigg( \frac{t}{\log L^2} \bigg)^{1/2}. 
\]
Then we have the corresponding sequence of maps $v_{t,L}(x) = (\sigma b_{t,L})^{-1}x - a_{t,L}$, with $v_{t,L}^{-1}$, $N_{t/K} \circ v_{t,L}^{-1}$, and $\mu_{t/K}^\nu \circ v_{t,L}^{-1}$ defined analogously to $v_{t}^{-1}$, $N_{t/K} \circ v_{t}^{-1}$, and $\mu_{t/K}^\nu \circ v_{t}^{-1}$. Note that $v_{t,L}^{-1}$ also satisfies \eqref{eq:zinvinterval} and \eqref{eq:zdisjoint}, and that 
\begin{equation}\label{eq:zLfirstorder}
	v_{t,L}^{-1}(x) = \sigma b_{t,L}(x + a_{t,L}) \sim \sigma \sqrt{2t\log L}, \qquad t\to\infty, 
\end{equation}
for every $x \in \Z$.

\subsection{Weak convergence to a Poisson random measure}\label{weakconvPRM}

All our limiting objects are Poisson random measures, and we write PRM$(\mu)$ to denote the law of a Poisson point process or Poisson random measure on $\R$ with intensity measure $\mu$. A random measure $M$ has the PRM$(\mu)$ distribution when, for any disjoint $A_1, \ldots, A_k \in \sB_\R$, the random variables $M(A_1), \ldots, M(A_k)$ are independent with $M(A_j) \sim \text{Poisson}\,(\mu(A_j))$. 

Let $\sM_\R$ denote the class of locally finite measures on $(\R, \sB_\R)$. Convergence of a sequence $\{\mu_n\} \subset \sM_\R$ is induced by the topology of vague convergence. Namely, $\mu_n \to \mu$ in $\sM_\R$ when $\int_\R f\,d\mu_n \to \int_\R f\,d\mu$ for all bounded, continuous $f : \R \to \R$. 

Then for a sequence of $\sM_\R$-valued random measures $\{M_n\}$, the convergence in distribution $M_n \Rightarrow M$ is characterized by $E[g(M_n)] \to E[g(M)]$ for all bounded $g : \sM_\R \to \R$ that are vaguely continuous. In fact, when $\{M_n\}, M$ are point process and $M$ is simple, then sufficient for $M_n \Rightarrow M$ is the convergence $M_n(A) \Rightarrow M(A)$ as $\R$-valued random variables for all sets $A$ in a sufficiently rich subclass $\sU \subset \sB_\R$ \cite[Theorem 4.15]{KalBook17}. 

In particular, $\sU$ can be any ring (closed under finite unions, finite intersections, and proper differences) that contains a countable topological basis for $\R$. In our context it is convenient to specify 
\[
	\sU = \left\{ \bigcup_{i=1}^n (a_i,b_i] : n \in \N, \; -\infty < a_i < b_i < \infty \right\}. 
\]

Because Poisson random measures are
simple, we have the following convergence criterion.

\begin{lemma}[Criteria for PRM convergence]\label{PPPcriteria} Let $\{\theta_t : t \ge 0\}$ be a sequence of point processes on $\R$ and $\mu$ a locally finite measure on $\R$. 
 The following are equivalent: 

	\begin{enumerate}[(a)]

	\item $\displaystyle \theta_t \underset{t\to\infty}{\Longrightarrow} \text{PRM}\,(\mu)$. 
	
	\item $\displaystyle \theta_t(B) \underset{t\to\infty}{\Longrightarrow} \text{Poisson}\,(\mu(B))$ for any $B \in \sU$. 
	
	\item $\displaystyle \theta_t(A_1) + \theta_t(A_2) + \cdots + \theta_t(A_m) \underset{t\to\infty}{\Longrightarrow} \text{Poisson}\,(\mu(A_1) + \mu(A_2) + \cdots + \mu(A_m))$ for any finite collection of disjoint, finite intervals $A_k = (a_k, b_k]$, $1 \le k \le m$. 	
	\end{enumerate}

\end{lemma}

The equivalence of (a) and (b) follows from the discussion preceding the lemma. (b) and (c) are seen to be equivalent by writing $B \in \sU$ as a finite union of disjoint intervals. We note that while the independence structure of the PRM limit is clear in (c), it is implicit in (b). 

We will make use of the criteria in both parts (b) and (c) of Lemma \ref{PPPcriteria}. As mentioned in the Introduction, we prove Poisson convergence via the moment method, using the factorial measures of $N_t$, which we now define.

 For a set $S$, let 
 \begin{equation}\label{eq:(n)notation}
 	S^{(n)} = \{(s_1, \ldots, s_n) \in S^n : s_i \ne s_j \;\text{for}\; i \ne j\}. 
\end{equation}
For a counting measure $\mu$ with atomic decomposition written $\mu = \sum_{i \in S} \delta_{x_i}$ with respect to $\{x_i\} \subset \R$ not necessarily distinct, 
its corresponding $n$th order factorial measure is the measure on $\R^n$ defined by
\begin{equation}\label{eq:factmzrdef}
	\mu^{(n)} = \sum_{(i_1, \ldots, i_n) \in S^{(n)}} \delta_{x_{i_1}, \ldots, x_{i_n}}. 
\end{equation}
For $k, n \in \N$, let $(k)_{n} = k(k-1) \cdots (k-n+1)$.
Let $A^n \subset \R^n$ denote the $n$-fold Cartesian product of $A \in \sB_\R$ with itself. A straight-forward computation connects $\mu^{(n)}$ to $(\mu)_n$ as given below.

\begin{lemma} If $\mu = \sum_{i \in S} \delta_{x_i}$ is a counting measure on $\R$ and $A \in \sB_\R$, then 
\begin{equation}\label{eq:factmzrasfactmoment}
	\mu^{(n)}(A^n) = \mu(A)(\mu(A) - 1) \cdots (\mu(A) - n + 1) = (\mu(A))_n. 
\end{equation}
\end{lemma}

\begin{proof} The identity is immediate for $n = 1$. Otherwise, if $\mu^{(n)}(A^n) = (\mu(A))_n$ then 
\begin{align*}
	(\mu(A))_{n+1} &= (\mu(A))_n (\mu(A) - n) = \mu^{(n)}(A^n) (\mu(A) - n) \\
	&= \sum_{(i_1, \ldots, i_n) \in S^{(n)}} \sum_{j \in S} 1(x_{i_1} \in A, \ldots, x_{i_n} \in A, x_j \in A) - n \mu^{(n)}(A^n) \\
	&= \sum_{(i_1, \ldots, i_n) \in S^{(n)}} \sum_{j \in S \setminus \{i_1, \ldots, i_n\}} 1(x_{i_1} \in A, \ldots, x_{i_n} \in A, x_j \in A) \\
	&\quad + \sum_{(i_1, \ldots, i_n) \in S^{(n)}} \sum_{j \in \{i_1, \ldots, i_n\}} 1(x_{i_1} \in A, \ldots, x_{i_n} \in A, x_j \in A) - n \mu^{(n)}(A^n) \\
	&= \mu^{(n+1)}(A^{n+1}) + n \mu^{(n)}(A^n) - n \mu^{(n)}(A^n) = \mu^{(n+1)}(A^{n+1}), 
\end{align*}
since $1(x_{i_1} \in A, \ldots, x_{i_n} \in A, x_j \in A) = 1(x_{i_1} \in A, \ldots, x_{i_n} \in A)$ when $j \in \{i_1, \ldots, i_n\}$. Now \eqref{eq:factmzrasfactmoment} follows by induction. 
\end{proof}

Given the previous lemma and the characterization of the Poisson distribution by its moments, we have the following equivalent formulations of the convergence in Lemma \ref{PPPcriteria} (b) and (c). 

\begin{lemma}[Criterion for Poisson convergence I]\label{PoissoncriterionI} Lemma \ref{PPPcriteria} (b) holds if and only if for any $B \in \sU$ and $n \in \N$,
\[
	E[\theta_t^{(n)}(B^n)] \underset{t\to\infty}{\longrightarrow} (\mu(B))^n. 
\]
\end{lemma}

\begin{lemma}[Criterion for Poisson convergence II]\label{PoissoncriterionII} Lemma \ref{PPPcriteria} (c) holds if and only if 
for any finite collection of disjoint, finite intervals $A_k = (a_k, b_k]$, $1 \le k \le m$, and any $n_1, n_2, \ldots, n_m \in \N$,
\[
	E\bigg[\prod_{k=1}^m \theta_t^{(n_k)}(A_k^{n_k})\bigg] \underset{t\to\infty}{\longrightarrow} \prod_{k=1}^m (\mu(A_k))^{n_k}. 
\]
\end{lemma}

\section{Main results}\label{results}

Here we give our main results. First we collect our assumptions on the jump probabilities as a reminder for the reader. 

\begin{condition}\label{lighttail} $p$ is irreducible and satisfies 
\begin{enumerate}[(i)]

\item $p(0,0) = 0$ and $p(x,y) = p(y,x) = p(0, y-x)$ for all $x, y \in \Z$. 

\item $\sum_{x \in \Z} e^{r x} p(0,x) < \infty$ for some $r > 0$. 

\end{enumerate}
\end{condition}

Recall also that $\sigma$ denotes the standard deviation of $p$. 
Next we specify the class of initial distributions $\nu$ considered. 

\begin{condition}\label{initialprofilecond} $\nu \in \sP_K^{\text{step}}$ such that 
\begin{equation}\label{eq:cnu}
	c_{\nu} =  \lim_{k \to \infty} \frac{1}{k} \sum_{x=0}^{k-1} E_{\nu}[\eta(-x)] \quad \text{exists and is positive.}
\end{equation}
\end{condition}

Above, \eqref{eq:cnu} serves as a definition of $c_\nu$ for a given $\nu \in \sP_K^{\text{step}}$ satisfying Condition \ref{initialprofilecond}.

\begin{remark}\upshape 
Any $\nu$ satisfying Condition \ref{initialprofilecond} has $\nu(\sum_{x \in \Z} \eta(x) = \infty) = 1$.
Notable subcases of \eqref{eq:cnu} are 
\begin{enumerate}[(i)]

\item $c_\nu = \lim_{x \to -\infty} E_\nu[\eta(x)]$ exists and is positive; 

\item $\nu$ is periodic, namely there is an integer $r > 0$ such that $E_\nu[\eta(x - r)] = E_\nu[\eta(x)]$ for all $x\le0$ and 
\[
	c_\nu = \frac{1}{r}\sum_{x=0}^{r-1} E_\nu[\eta(-x)] > 0. 
\]

\item When $\nu = \delta_{\eta^\infty}$, then $c_\nu = K$.

\end{enumerate} 
\end{remark}

\subsection{Point process limit}

Our main results establish convergence of $N_{t/K} \circ v^{-1}_t$ or $N_{t/K} \circ v_{t,L}^{-1}$ to a Poisson random measure with intensity proportional to
\begin{equation}
	\lambda(dx) := e^{-x}\,dx, 
\end{equation}
with proportionality constants determined by $\{L\}$ and initial distribution $\eta_0 \sim \nu$.

The first theorem considers the simplest case of the `full step' deterministic initial profile $\eta^\infty$ defined in \eqref{eq:fullstep} (recall the shorthand $\P_\infty = \P_{\eta^\infty}$). 

\begin{theorem}\label{main} Under Condition \ref{lighttail}, 
\[
	\P_\infty\big( N_{t/K}\circ v^{-1}_t \in \cdot \big) \underset{t\to\infty}{\Longrightarrow} \text{PRM}\,(K\sigma\lambda). 
\]
\end{theorem}

Next is our more general result
(Theorem \ref{main} corresponds to part (a) of Theorem \ref{Lmain} with $L \equiv \infty$ and $\psi = \infty$). For $\nu \in \sP_K^{\text{step}}$, recall the truncated measure $\nu_L$ defined in \eqref{eq:Lmeasure}. 

\begin{theorem}\label{Lmain} Assume Condition \ref{lighttail}, and suppose $\nu \in \sP_K^{\text{step}}$ satisfies Condition \ref{initialprofilecond}. 
Consider a sequence $\{L\} \subset \N \cup \{\infty\}$ with $L \uparrow \infty$ as $t \to \infty$. 

\begin{enumerate}[(a)]

\item If $L \big(\frac{\log t}{t}\big)^{1/2} \to \psi \in (0, \infty]$, then 
\[
	\P_{\nu_L}\big(N_{t/K}\circ v_t^{-1} \in \cdot \big) \underset{t\to\infty}{\Longrightarrow} \text{PRM}\,\big(c_\nu\sigma(1 - e^{-\psi/\sigma})\lambda\big). 
\]

\item If $L \big(\frac{\log t}{t}\big)^{1/2} \to 0$, then 
\[
	\P_{\nu_L}\big(N_{t/K} \circ v_{t,L}^{-1} \in \cdot \big) \underset{t\to\infty}{\Longrightarrow} \text{PRM}\,(c_\nu\lambda). 
\]

\end{enumerate}
\end{theorem}

\begin{remark}\upshape
\begin{enumerate}[(i)]

\item
Suppose $\nu \in \sP_K^{\text{step}}$ has binomial marginals on $\{\ldots, -2, -1, 0\}$. More precisely, for some $\alpha \in (0,1)$, 
\[
	\nu(\eta(x) \in \cdot) = \begin{cases} \delta_0 & x > 0 \\
								\text{Binomial}\,(K,\alpha) & x \le 0. \end{cases}
\]
Then the convergence 
\begin{equation}\label{eq:thinned}
	\P_\nu\big(N_{t/K} \circ v_t^{-1} \in \cdot\big) \underset{t\to\infty}{\Longrightarrow} \text{PRM}\,(\alpha K\sigma\lambda)
\end{equation}
follows immediately from Theorem \ref{main} by the following thinning argument. The law of $N_{t/K}$ under $\P_\nu$ is just an i.i.d. thinning of its law under $\P_\infty$. Namely, each stirring variable $\xi_t^{xj}$ is removed from the point processes independently with probability $1 - \alpha$. By the stirring construction, this does not affect the joint distributions of the remaining stirring variables, so by Lemma 4.17 of \cite{KalBook17}, the convergence in Theorem \ref{main} implies the convergence for the corresponding thinned processes in \eqref{eq:thinned}. 

\item The assumption that $\nu \in \sP_K^{\text{step}}$, in particular that $\nu$ puts no weight on $\{1, 2, 3, \ldots\}$, can be easily generalized to $\nu \in \sP_K$ with $\sum_{x > 0} E_\nu[\eta(x)] < \infty$, which guarantees $X_t^{(0)} < \infty$ a.s. For such a $\nu$, consider its truncated version $\nu_\infty \in \sP_K^{\text{step}}$, 
and note that $\nu_\infty \le \nu$ in the sense of stochastic ordering. 
For any bounded $A \in \sB_\R$, \eqref{eq:zfirstorder}, \eqref{eq:zinvinterval}, and the central limit theorem imply 
$P_x(\zeta_t \in v_t^{-1}(A)) \to 0$ for all $x \in \Z$. Then by the dominated convergence theorem 
(cf. \eqref{eq:meanrep}),  
\[
	0 \le 
	\mu^\nu_{t/K} \circ v_t^{-1}(A) - \mu^{\nu_\infty}_{t/K} \circ v_t^{-1}(A) 
	= \sum_{x > 0} E_\nu[\eta(x)]P_x(\zeta_t \in v_t^{-1}(A)) \to 0, 
\]
which implies that $\P_{\nu_{\infty}}(N_{t/K} \circ v_t^{-1}(A) \in \cdot)$ and $\P_\nu(N_{t/K} \circ v_t^{-1}(A) \in \cdot)$ converge to the same Poisson distribution. In particular this holds for any $A$ in the ring $\sU$ specified in Section \ref{weakconvPRM}, which means that $N_{t/K} \circ v_t^{-1}$ converges in distribution to the same Poisson random measure under $\P_\nu$ and $\P_{\nu_\infty}$. From \eqref{eq:zLfirstorder}, the same argument applies with $v_t$ replaced by $v_{t,L}$. 

\end{enumerate}
\end{remark}

\subsection{Limit theorems for order statistics}\label{OSlimits}

Let $\{\chi_n : n \ge 0\}$ be i.i.d. Exponential$\,(1)$ random variables defined on a space with probability measure $P$, and let $T_n = \chi_0 + \cdots + \chi_n$, $n \ge 0$. 
Suppose $N_\mu \sim \text{PRM}\,(\mu)$ has a maximum point, i.e., that $G_\mu(x) := \mu(x, \infty) < \infty$ for any $x \in \R$. 
Then the transformed process $N_\mu \circ G_\mu^{-1}$ is a rate-$1$ homogeneous Poisson process on $[0, \infty)$, so it has the representation $N_\mu \circ G_\mu^{-1} \overset{d}{=} \sum_{n=0}^\infty \delta_{T_n}$. 

The limiting objects in Theorems \ref{main} and \ref{Lmain} are Poisson random measures $N_{c\lambda}$ for a constant $c > 0$, and $G_{c\lambda}(x) = c\lambda(x, \infty) = ce^{-x}$. 
It follows from the discussion in the previous paragraph that
\begin{equation}\label{eq:N_crep}
	N_{c\lambda} \overset{d}{=} \sum_{n=0}^\infty \delta_{-\log(T_n/c)}, 
\end{equation}
where $-\log(T_0/c) > - \log(T_1/c) > \cdots$ a.s.

Then we have the following results concerning limits of order statistics. See Corollary 4.1 in \cite{MikYsl2020} for a closely related result concerning the order statistics of independent, discrete-time random walks.

\begin{corollary}\label{orderstatlimit} Suppose $p$ satisfies Condition \ref{lighttail}, $\nu$ satisfies Condition \ref{initialprofilecond}, and $\eta_0 \sim \nu_L$, where $L \uparrow \infty$ as $t \to \infty$. Let $0 \le m_0 < m_1 < \cdots < m_k$. 

\begin{enumerate}[(a)]

\item If $L \big(\frac{\log t}{t}\big)^{1/2} \to \psi \in (0, \infty]$, then 
\begin{align*}
	&\bigg( \frac{X^{(m_0)}_{t/K}}{\sigma b_t} - a_t, \ldots,  \frac{X^{(m_k)}_{t/K}}{\sigma b_t} - a_t \bigg) \underset{t\to\infty}{\Longrightarrow} \bigg( - \log  \frac{T_{m_0}}{c_\nu\sigma(1 - e^{-\psi/\sigma})}, \ldots, - \log  \frac{T_{m_k}}{c_\nu\sigma(1 - e^{-\psi/\sigma})} \bigg). 
\end{align*}

\item If $L \big(\frac{\log t}{t}\big)^{1/2} \to 0$, then 
\begin{align*}
	&\bigg( \frac{X^{(m_0)}_{t/K}}{\sigma b_{t,L}} - a_{t,L}, \ldots,  \frac{X^{(m_k)}_{t/K}}{\sigma b_{t,L}} - a_{t,L} \bigg) \underset{t\to\infty}{\Longrightarrow} \bigg( - \log  \frac{T_{m_0}}{c_\nu}, \ldots, - \log  \frac{T_{m_k}}{c_\nu} \bigg). 
\end{align*}

\end{enumerate}
\end{corollary}

\begin{proof} 
We show part (a); (b) is similar.  
The convergence $N_{t/K} \circ v_t^{-1}\Rightarrow N_{c\lambda}$ provided by Theorem \ref{Lmain}, where $c = c_\nu\sigma(1 - e^{-\psi/\sigma})$ and $N_{c\lambda}$ is given in \eqref{eq:N_crep}, implies the joint convergence 
\begin{equation}\label{eq:jointconv}
	(N_{t/K}\circ v_t^{-1}(A_1) , \ldots, N_{t/K} \circ v_t^{-1}(A_k) ) \Rightarrow (N_{c\lambda}(A_1), \ldots, N_{c\lambda}(A_k)), 
\end{equation}
for any $A_1, \ldots, A_k \in \sB_\R$. 

Note from \eqref{eq:NdefOS} that $X_{t}^{(m)} \le x$ if and only if $N_t(x, \infty) \le m$, and therefore $v_t(X_{t/K}^{(m)}) \le x$ if and only if $N_{t/K} \circ v_t^{-1}(x,\infty) \le m$. 
Then for any $x_0 \ge \cdots \ge x_k$ in $\R$ we have from \eqref{eq:jointconv} that 
\begin{align*}
	&\P_{\nu_L}(v_t(X_{t/K}^{(m_0)}) \le x_0, \ldots, v_t(X_{t/K}^{(m_k)}) \le x_k ) \\
	&= \P_{\nu_L}( N_{t/K} \circ v_t^{-1}(x_0,\infty) \le m_0, \ldots, N_{t/K} \circ v_t^{-1}(x_k,\infty) \le m_k ) \\
	&\to P(N_{c\lambda}(x_0, \infty) \le m_0, \ldots, N_{c\lambda}(x_k, \infty) \le m_k) \\
	&= P(-\log(T_{m_0}/c) \le x_0, \ldots, -\log(T_{m_k}/c) \le x_k). \qedhere
\end{align*}
\end{proof}

As a specific case of the previous corollary, the maximal particle position converges to a Gumbel distribution: In the context of part (a), 
\[
	\lim_{t\to\infty}\P_{\nu_L}\bigg( \frac{X^{(0)}_{t/K}}{\sigma b_t} - a_t \le x \bigg) = \exp\big( - c_\nu\sigma (1 - e^{-\psi/\sigma}) e^{-x} \big), \qquad x\in \R, 
\]
and in the context of part (b), 
\[
	\lim_{t\to\infty}\P_{\nu_L}\bigg( \frac{X^{(0)}_{t/K}}{\sigma b_t} - a_t \le x \bigg) = \exp\big( - c_\nu e^{-x} \big), \qquad x\in \R. 
\]

From Corollary \ref{orderstatlimit} and the continuous mapping theorem, we immediately obtain the joint limiting distributions of spacings between particles: 

\begin{corollary}\label{spacings} Suppose $p$ satisfies Condition \ref{lighttail}, $\nu$ satisfies Condition \ref{initialprofilecond}, and $\eta_0 \sim \nu_L$, where $L \uparrow \infty$ as $t \to \infty$. If $\lim_{t\to\infty} L\big(\frac{\log t}{t}\big)^{1/2} > 0$, then for any $m \ge 1$, 
\begin{equation}\label{eq:spacingdist}
	\bigg( \frac{X^{(0)}_{t/K} - X^{(1)}_{t/K}}{\sigma b_t}, 
	\ldots, \frac{X^{(m-1)}_{t/K} - X^{(m)}_{t/K}}{\sigma b_t} \bigg) \underset{t\to\infty}{\Longrightarrow} \bigg( \log \frac{T_1}{T_0}, 
	\ldots, \log \frac{T_m}{T_{m-1}} \bigg), 
\end{equation}
where $\{\log\big(\frac{T_k}{T_{k-1}}\big) : k \ge 1\}$ are independent and $\log\big(\frac{T_k}{T_{k-1}}\big) \sim \text{Exponential}\,(k)$. If $L\big(\frac{\log t}{t}\big)^{1/2} \to 0$, then \eqref{eq:spacingdist} holds with $b_{t,L}$ in place of $b_t$. 
\end{corollary}

\begin{remark} \upshape 
\begin{enumerate}[(i)]

\item The limiting distribution in \eqref{eq:spacingdist} matches that for the spacings of an eventually-infinite collection of independent, discrete-time random walks starting from $0$ (with different prelimit scaling) \cite{MikYsl2020}. 
Moreover, the $\bigotimes_{k \ge 1} \text{Exp}\,(k)$ distribution is stationary for the gaps in a collection of infinitely many independent Brownian motions with a maximum \cite[Example 1.8]{SarTsa2017}. 

\item Notably these limiting spacing distributions do not depend on the initial condition $\nu$. 
The intuition is thus: The assumption \eqref{eq:cnu} on $\nu$ ensures both that there are (eventually) infinitely many particles in the system and that the average starting distance between order statistics is approximately uniform. Heuristically, this uniformity means that after enough time the exact starting locations of particles is forgotten. 

\end{enumerate}
\end{remark}

\section{Proofs of main results}
\label{pfoutline}

Recall Lemmas \ref{PPPcriteria}, \ref{PoissoncriterionI}, and \ref{PoissoncriterionII} and the ring $\sU \subset \sB_\R$ defined in Section \ref{weakconvPRM}. Proofs of Theorems \ref{main} and \ref{Lmain} are given after an outline of their methods. The proofs of results stated in the outline are provided in Section \ref{otherproofs}.

\subsection{Proof outline for Theorems \ref{main} and \ref{Lmain}}\label{outline}

Theorem \ref{main} is proved first. While its result is technically covered by Theorem \ref{Lmain} (a), it turns out that a convenient way to prove the more general Theorem \ref{Lmain} is by comparison of an initial condition to the `full step' setting of Theorem \ref{main}. 

For both theorems, the sufficient factorial moment convergence will be shown. Recall the definition of the factorial measure in \eqref{eq:factmzrdef}. In particular, note that $N^{(1)}_t = N_t$, so that $\E_\nu[N^{(1)}_t(A)] = \mu_t^\nu(A)$ for $A \in \sB_\R$. Thus we begin by showing that the mean measures converge appropriately on $\sU$.  By additivity, it suffices to show the result for finite intervals of the form $(a, b]$. Recall that $f(t) \sim g(t)$ as $t \to \infty$ means $\lim_{t\to\infty} f(t)/g(t) = 1$.

\begin{prop}[Mean convergence]\label{meanconv} Let $A = (a,b]$ for $-\infty < a < b < \infty$, suppose $p$ satisfies Condition \ref{lighttail},
and suppose $\nu$ satisfies Condition \ref{initialprofilecond}. Then if $L \uparrow \infty$ as $t \to \infty$ and $w_t$ denotes either $v_t$ or $v_{t,L}$ for all $t$, 
\begin{equation}\label{eq:nuLvsLmean}
	\mu_{t/K}^{\nu_L} \circ w_t^{-1}(A)  \sim (c_\nu/K) \mu_{t/K}^{L} \circ w_t^{-1}(A) , \qquad t \to \infty. 
\end{equation}
In particular, we have the following. 
\begin{enumerate}[(a)]

\item If $L \big(\frac{\log t}{t}\big)^{1/2} \to \psi \in (0, \infty]$, then 
\[
	\mu^{\nu_L}_{t/K} \circ v_t^{-1}(A) \underset{t\to\infty}{\longrightarrow} c_\nu\sigma(1 - e^{-\psi/\sigma})\lambda(A). 
\]
\item If $L \big(\frac{\log t}{t}\big)^{1/2} \to 0$, then 
\[
	\mu^{\nu_L}_{t/K} \circ v_{t,L}^{-1}(A)  \underset{t\to\infty}{\longrightarrow} c_\nu \lambda(A). 
\]
\end{enumerate}
\end{prop}

Now recall that $\zeta_t$ is a continuous-time random walk on $\Z$ with jump distribution $p$, and that $\{P_x : x \in \Z\}$ are defined by $P_x(\zeta_0 = x) = 1$. The following quantities appear in subsequent error estimates. 

\begin{definition}\label{kappataudef} For $A, B \in \sB_\R$, 
\begin{align*}
	\tau_t(A,B) &= \sum_{x \in A \cap \Z} P_x(\zeta_t \in B)^2, \qquad \text{and} \\
	\kappa_t(A,B) &= \sum_{x,y \in \Z} p(x,y) \int_0^t (P_x(\zeta_s \in B) - P_y(\zeta_s \in B))^2 P_x(\zeta_{t-s} \in A)P_y(\zeta_{t-s} \in A)  \,ds.
\end{align*}
\end{definition}

Below we give a bound on the difference of powers of the mean measure and the joint factorial measure moments for the point process $N_t$, in the context of a deterministic initial profile $\eta$ for which $\eta(x) \in \{0, K\}$ for every $x$. Note that every $\eta^L$, $L \in \N\cup\{\infty\}$, as in \eqref{eq:fullstep} is of this type. Along with Proposition \ref{meanconv}, vanishing of the upper bound in Proposition \ref{factmomentbound} will verify the criterion of Lemma \ref{PoissoncriterionI} for such a deterministic initial condition. 

\begin{prop}[Deterministic profile joint factorial moments]\label{factmomentbound} Suppose that $A_k = (a_k,b_k]$, $1 \le k \le m$, are finite and disjoint, and define $A \subset \R$ by
\begin{equation}\label{eq:bigA}
	A = ( \min\{a_1, \ldots, a_m\}, \max\{b_1, \ldots, b_m\}]. 
\end{equation}
Then for any $\eta \in \{0, K\}^\Z \subset \sX_K$ and $n_1, \ldots, n_m \in \N$, 
\begin{align*}
	0&\le  \prod_{k=1}^m ( \mu^\eta_{t/K}(A_k) )^{n_k} - \E_{\eta}\bigg[ \prod_{k=1}^m N_{t/K}^{(n_k)}(A_k^{n_k}) \bigg] \\
	&\le K^2 {n_1 + \cdots + n_m \choose 2} ( \mu_{t/K}^\eta(A) )^{n_1 + \cdots + n_m - 2} ( \kappa_t(A, \{\eta > 0\}) + \tau_t(A, \{\eta > 0\}) ). 
\end{align*}
\end{prop}

Proposition \ref{factmomentbound} will later be applied for the different scaling regimes of Theorem \ref{Lmain} (a) and (b). We now state that in either case the errors vanish in the $t \to \infty$ limit. 

\begin{prop}[Error estimate]\label{kappatauconv} Let $A \in \sU$. 
\begin{enumerate}[(a)] 

\item We have 
\[
	\lim_{t\to\infty} \kappa_t(v_t^{-1}(A), (-\infty,0]) = 0. 
\]
Moreover, if $L \uparrow \infty$ such that $L(\frac{\log t}{t})^{1/2} \to 0$ as $t \to \infty$, then 
\[
	\lim_{t\to\infty} \kappa_t(v_{t,L}^{-1}(A), (-L,0]) = 0. 
\]

\item If $\lim_{t\to\infty} L(\frac{\log t}{t})^{1/2} > 0$, then 
\[
	\lim_{t\to\infty} \tau_t(v_t^{-1}(A), (-L, 0]) = \lim_{t\to\infty} \tau_t((-L,0], v_t^{-1}(A)) = 0. 
\]
Moreover, if $L \uparrow \infty$ such that $L(\frac{\log t}{t})^{1/2} \to 0$ as $t \to \infty$, then 
\[
	\lim_{t\to\infty} \tau_t(v_{t,L}^{-1}(A), (-L, 0]) = \lim_{t\to\infty} \tau_t((-L,0], v_{t,L}^{-1}(A)) = 0. 
\]

\end{enumerate}

\end{prop}

We will show the factorial moment convergence in Lemma \ref{PoissoncriterionI} for a product measure initial condition by comparison to a deterministic profile, using the following. For such a measure $\nu$, recall the notation $\{\nu > 0\} = \{x \in \Z : \nu(\eta(x) > 0) > 0\}$.

\begin{prop}[Product measure factorial moments]\label{genmzrbound} Let $A \in \sB_\R$ and $n \in \N$. If $\eta, \chi \in \sX_K$ such that $\eta(x) \le \chi(x)$ for all $x \in \Z$, then  
\begin{equation}\label{eq:icmonotonicity}
	0 \le (\mu_t^\eta(A))^n - \E_\eta[N_t^{(n)}(A^n)] \le (\mu_t^{\chi}(A))^n - \E_{\chi}[N_t^{(n)}(A^n)]. 
\end{equation}
Moreover, if $\nu \in \sP_K$ and we let $\chi_\nu = K1_{\{\nu > 0\}}$, then  
\begin{equation}\label{eq:icmonotonicity2}
\begin{aligned}
	&-K^2 {n \choose 2} (\mu_{t/K}^{\chi_\nu}(A))^{n-2} \tau_{t}(\{\nu > 0\}, A) \\
	&\le ( \mu_{t/K}^\nu(A) )^n - \E_{\nu}[N_{t/K}^{(n)}(A^n)] \le (\mu_{t/K}^{\chi_\nu}(A))^n - \E_{\chi_\nu}[N_{t/K}^{(n)}(A^n)]. 
\end{aligned}
\end{equation}
\end{prop}

\subsection{Proof of Theorem \ref{main}}

Recall the assumed deterministic initial profile $\eta^\infty = K1_{(-\infty,0]}$. 
Let $A_k = (a_k, b_k]$, $1 \le k \le m$, be disjoint finite intervals. From Lemmas \ref{PPPcriteria} and \ref{PoissoncriterionII}, to prove Theorem \ref{main} it suffices to show that 
\begin{equation}\label{eq:toshowformain}
	\lim_{t \to \infty} \E_\infty \bigg[ \prod_{k=1}^m (N_{t/K} \circ v_t^{-1})^{(n_k)}(A_k^{n_k}) \bigg] = \prod_{k=1}^m (K\sigma \lambda(A_k))^{n_k}, 
\end{equation}
for any $n_1, \ldots, n_m \in \N$. We establish \eqref{eq:toshowformain} in two steps. 

\medskip

\emph{Step 1.} 
From Proposition \ref{meanconv} (a),  
\[
	\lim_{t\to\infty} \mu_{t/K}^\infty \circ v_t^{-1}(A_k) = K \sigma \lambda(A_k), 
\]
for all $1\le k \le m$. Consequently, 
\begin{equation}\label{eq:step1result}
	\lim_{t\to\infty} \prod_{k=1}^m (\mu_{t/K}^\infty \circ v_t^{-1}(A_k))^{n_k} = \prod_{k=1}^m (K\sigma \lambda(A_k))^{n_k}, 
\end{equation}
for any $n_1, \ldots, n_m \in \N$. 

\medskip

\emph{Step 2.} Fix $n_1, \ldots, n_m \in \N$. The second step is to show 
\begin{equation}\label{eq:fullsteptoshow}
	\prod_{k=1}^m (\mu_{t/K}^\infty \circ v_t^{-1}(A_k))^{n_k} - \E_\infty \bigg[ \prod_{k=1}^m (N_{t/K} \circ v_t^{-1})^{(n_k)}(A_k^{n_k}) \bigg] \underset{t\to\infty}{\longrightarrow} 0, 
\end{equation}
which by \eqref{eq:step1result}
implies the limit \eqref{eq:toshowformain}.

Let $A = (\min_{1\le k \le m} a_k, \max_{1 \le k \le m} b_k]$. Since $v_t$ is an increasing and one-to-one function,  
\[
	 v_t^{-1}(A) = \Big( \min_{1 \le k \le m} v_t^{-1}(a_k), \max_{1 \le k \le m} v_t^{-1}(b_k)  \Big]. 
\]
Note also from \eqref{eq:factmzrasfactmoment} that 
\begin{equation}\label{eq:factmeasurez}
	(N_{t/K} \circ v_t^{-1})^{(n_k)}(A_k^{n_k}) = (N_{t/K} \circ v_t^{-1}(A_k))_{n_k} = N_{t/K}^{(n_k)}((v_t^{-1}(A_k))^{n_k}), 
\end{equation}
for each $k$. 
Moreover, $\{\eta^\infty > 0 \} = (-\infty, 0] \cap \Z$, and $v_t^{-1}(A_k) = (v_t^{-1}(a_k), v_t^{-1}(b_k)]$, $1 \le k \le m$, are disjoint (cf.  \eqref{eq:zdisjoint}). Then, applying Proposition \ref{factmomentbound} yields 
\begin{align*}
	0 &\le \prod_{k=1}^m (\mu_{t/K}^\infty \circ v_t^{-1}(A_k))^{n_k} - \E_\infty \bigg[ \prod_{k=1}^m (N_{t/K} \circ v_t^{-1})^{(n_k)}(A_k^{n_k}) \bigg] \\
	&= \prod_{k=1}^m (\mu_{t/K}^\infty (v_t^{-1}(A_k)))^{n_k} - \E_\infty \bigg[ \prod_{k=1}^m N_{t/K}^{(n_k)}((v_t^{-1}(A_k))^{n_k}) \bigg] \\
	&\le K^2 {n_1 + \cdots + n_m \choose 2} (\mu_{t/K}^\infty \circ v_t^{-1}(A))^{n_1 + \cdots + n_m - 2} \\
	&\qquad\qquad \times (\kappa_t(v_t^{-1}(A), (-\infty, 0]) + \tau_t(v_t^{-1}(A), (-\infty, 0])). 
\end{align*}
Now, $\mu_{t/K}^\infty \circ v_t^{-1}(A) \to K \sigma \lambda(A)$ as $t \to \infty$ from Proposition \ref{meanconv} (a), and 
\[
	\kappa_t(v_t^{-1}(A), (-\infty, 0]) + \tau_t(v_t^{-1}(A), (-\infty, 0]) \underset{t\to\infty}{\longrightarrow} 0
\]
from Proposition \ref{kappatauconv}. This shows \eqref{eq:fullsteptoshow} and consequently \eqref{eq:toshowformain}, completing the proof of Theorem \ref{main}. \hfil\qed

\subsection{Proof of Theorem \ref{Lmain}}

We prove part (a) of Theorem \ref{Lmain} followed by part (b). 

\medskip

\textbf{Proof of (a).} Recall that $\nu \in \sP_K^{\text{step}}$ and that $L(\frac{\log t}{t})^{1/2} \to \psi \in (0, \infty]$. The proof relies on a comparison of the truncated initial distribution $\nu_L$ defined in \eqref{eq:Lmeasure} to the determinisitc `full step' profile $\eta^\infty$, using the result of Theorem \ref{main}. 

Let $B \in \sU$ and $n \in \N$. By Lemmas \ref{PPPcriteria} and \ref{PoissoncriterionI}, it suffices to show that 
\begin{equation}\label{eq:toshowforL(a)}
	\lim_{t \to \infty} \E_{\nu_L}[(N_{t/K} \circ v_t^{-1})^{(n)}(B^n)] = (c_\nu\sigma(1 - e^{-\psi/\sigma})\lambda(B))^n. 
\end{equation}
By these same lemmas, Theorem \ref{main} implies that 
\begin{equation}\label{eq:bymain}
	\lim_{t \to \infty} \E_{\infty}[(N_{t/K} \circ v_t^{-1})^{(n)}(B^n)] = \lim_{t\to\infty} (\mu_{t/K}^\infty \circ v_t^{-1}(B))^n = (K\sigma\lambda(B))^n.
\end{equation}
We establish \eqref{eq:toshowforL(a)} in two steps. 

\medskip

\emph{Step 1.} By decomposing $B$ into disjoint intervals and using additivity, Proposition \ref{meanconv} (a) implies 
\begin{equation}\label{eq:L(a)mean}
	\lim_{t\to\infty} \mu^{\nu_L}_{t/K} \circ v_t^{-1}(B) = c_\nu \sigma(1 - e^{-\psi/\sigma})\lambda(B). 
\end{equation}

\medskip

\emph{Step 2.} The second step is to show that 
\begin{equation}\label{eq:prodfullsteptoshow}
	(\mu_{t/K}^{\nu_L} \circ v_t^{-1}(B))^n - \E_{\nu_L}[(N_{t/K} \circ v_t^{-1})^{(n)}(B^n)] \underset{t\to\infty}{\longrightarrow} 0, 
\end{equation}
which together with \eqref{eq:L(a)mean} implies \eqref{eq:toshowforL(a)}.

Recalling the notation of Section \ref{stepprofiles}, let $\chi^L = K1_{\{\nu_L > 0\}}$. Then $\{\chi^L > 0\} \subset (-L, 0]$ and $\chi^L(x) \le \eta^\infty(x)$ for all $x \in \Z$. Applying \eqref{eq:icmonotonicity2} from Proposition \ref{genmzrbound}, noting \eqref{eq:factmeasurez}, we have 
\begin{equation}\label{eq:Lstep2bound}
\begin{aligned}
	&\big| (\mu_{t/K}^{\nu_L} \circ v_t^{-1}(B))^n - \E_{\nu_L}[(N_{t/K} \circ v_t^{-1})^{(n)}(B^n)] \big| \\
	&\le  (\mu_{t/K}^{\chi^L} \circ v_t^{-1}(B))^n - \E_{\chi^L}[(N_{t/K} \circ v_t^{-1})^{(n)}(B^n)] \\
	&\quad + K^2{n \choose 2} (\mu_{t/K}^{\chi^L} \circ v_t^{-1}(B))^{n-2} \tau_t(\{\nu_L > 0\}, v_t^{-1}(B)). 
\end{aligned}
\end{equation}
Moreover, by both \eqref{eq:icmonotonicity} of Proposition \ref{genmzrbound} and \eqref{eq:bymain}, 
\begin{equation}\label{eq:Lstep2first}
\begin{aligned}
	&(\mu_{t/K}^{\chi^L} \circ v_t^{-1}(B))^n - \E_{\chi^L}[(N_{t/K} \circ v_t^{-1})^{(n)}(B^n)] \\
	&\le (\mu_{t/K}^{\infty} \circ v_t^{-1}(B))^n - \E_{\infty}[(N_{t/K} \circ v_t^{-1})^{(n)}(B^n)] \underset{t\to\infty}{\longrightarrow} 0. 
\end{aligned}
\end{equation}
Next, note that $\mu_t^{\eta}$ is monotone in $\eta$ and $\tau_t$ is monotone in both its entries. Since $\{\nu_L > 0\} = \{\chi^L > 0\} \subset (-L, 0]$, 
\eqref{eq:bymain} and Proposition \ref{kappatauconv} (b) imply 
\begin{equation}\label{eq:Lstep2second}
\begin{aligned}
	(\mu_{t/K}^{\chi^L} \circ v_t^{-1}(B))^{n-2} \tau_t(\{\nu_L > 0\}, v_t^{-1}(B)) &\le (\mu_{t/K}^\infty \circ v_t^{-1}(B))^{n-2} \tau_t((-L, 0], v_t^{-1}(B)) \\
	&\underset{t\to\infty}{\longrightarrow} 0. 
\end{aligned}
\end{equation}
\eqref{eq:Lstep2first} and \eqref{eq:Lstep2second} show that the upper bound in \eqref{eq:Lstep2bound} vanishes as $t \to \infty$, establishing \eqref{eq:prodfullsteptoshow} and thereby completing the proof of Theorem \ref{Lmain} (a). \hfil\qed

\medskip

\textbf{Proof of (b).}
This is proved in a similar manner to Theorem \ref{main} and Theorem \ref{Lmain} (a). Recall that $L \uparrow \infty$ and $L(\frac{\log t}{t})^{1/2} \to 0$ as $t \to \infty$. 

First consider the deterministic initial profile $\eta_L$ defined in \eqref{eq:fullstep}. As in the proof of Theorem \ref{main}, it suffices to show that 
\begin{equation}\label{eq:toshowforL(b)det}
	\lim_{t \to \infty} \E_L \bigg[ \prod_{k=1}^m (N_{t/K} \circ v_{t,L}^{-1})^{(n_k)}(A_k^{n_k}) \bigg] = \prod_{k=1}^m (K\lambda(A_k))^{n_k}, 
\end{equation}
for any disjoint, finite intervals $A_1 = (a_1, b_1], \ldots, A_m = (a_m, b_m]$ and any $n_1, \ldots, n_m \in \N$. By Proposition \ref{meanconv} (b), 
\begin{equation}\label{eq:step1resultL(b)det}
	\lim_{t\to\infty} \prod_{k=1}^m (\mu_{t/K}^L \circ v_{t,L}^{-1}(A_k))^{n_k} = \prod_{k=1}^m (K \lambda(A_k))^{n_k}. 
\end{equation}
Letting $A = (\min_k a_k, \max_k b_k]$, Proposition \ref{factmomentbound}, Proposition \ref{meanconv} (b), and Proposition \ref{kappatauconv} imply
\begin{align*}
	0 &\le \prod_{k=1}^m (\mu_{t/K}^L \circ v_{t,L}^{-1}(A_k))^{n_k} - \E_L \bigg[ \prod_{k=1}^m (N_{t/K} \circ v_{t,L}^{-1})^{(n_k)}(A_k^{n_k}) \bigg] \\
	&\le K^2 {n_1 + \cdots + n_m \choose 2} (\mu_{t/K}^L \circ v_{t,L}^{-1}(A))^{n_1 + \cdots + n_m - 2} \\
	&\qquad\qquad \times (\kappa_t(v_{t,L}^{-1}(A), (-L, 0]) + \tau_t(v_{t,L}^{-1}(A), (-L, 0])) \underset{t\to\infty}{\longrightarrow} 0. 
\end{align*} 
Along with with \eqref{eq:step1resultL(b)det}, the above establishes \eqref{eq:toshowforL(b)det} and proves Theorem \ref{Lmain} (b) in the case of determinisitic initial profile $\eta_L$.

The proof for more general product initial measure $\nu_L$ rests on the result for determinisitc $\eta_L$, following the above scheme for proving Theorem \ref{Lmain} (a) using Theorem \ref{main}. In particular, for $B \in \sU$, 
\begin{equation}\label{eq:L(b)mean}
	\lim_{t\to\infty} \mu^{\nu_L}_{t/K} \circ v_{t,L}^{-1}(B) = c_\nu \lambda(B), 
\end{equation}
by Proposition \ref{meanconv} (b). Then, for $n \in \N$, by Propostion \ref{genmzrbound}, Proposition \ref{kappatauconv}, the result for $\eta_L$, and the monotonicity of $\mu^\eta_t$ and $\tau_t$ used in Step 2 of the proof of Theorem \ref{Lmain} (a), 
\begin{equation}\label{eq:L(b)factbound}
\begin{aligned}
	&\big| (\mu_{t/K}^{\nu_L} \circ v_{t,L}^{-1}(B))^n - \E_{\nu_L}[(N_{t/K} \circ v_{t,L}^{-1})^{(n)}(B^n)] \big| \\
	&\le  (\mu_{t/K}^{L} \circ v_{t,L}^{-1}(B))^n - \E_{L}[(N_{t/K} \circ v_{t,L}^{-1})^{(n)}(B^n)] \\
	&\quad + K^2{n \choose 2} (\mu_{t/K}^{L} \circ v_{t,L}^{-1}(B))^{n-2} \tau_t((-L, 0], v_{t,L}^{-1}(B)) \underset{t\to\infty}{\longrightarrow} 0. 
\end{aligned}
\end{equation}
Together, \eqref{eq:L(b)mean} and \eqref{eq:L(b)factbound} show 
\[
	\lim_{t\to\infty} \E_{\nu_L}[(N_{t/K} \circ v_{t,L}^{-1})^{(n)}(B^n)] = (c_\nu \lambda(B))^n, 
\]
which completes the proof of Theorem \ref{Lmain} (b) by the criteria from Lemmas \ref{PPPcriteria} and \ref{PoissoncriterionI}. \hfil\qed

\section{Semigroup and correlation inequalities}\label{semigroup}

Throughout this section, $p$ is a symmetric transition kernel on $\Z$, but it is not assumed to be translation invariant. 

Fix $n \ge 1$, and consider a $K$-SEP system of $n$ particles with positions at time $t$ labeled $Y^1_t, Y_t^2, \ldots, Y_t^n$. 
The process of particle positions has state space 
\begin{equation}\label{eq:Omegadef}
	\Omega^K_n := \bigg\{\bx \in \Z^n : \sum_{j=1}^n 1(x_i = x_j) \le K \;\text{ for all }\; 1 \le i \le n \bigg\}, 
\end{equation}
and generator
\[
	\sV_n^Kf(\bx) = \sum_{i=1}^n \sum_{y \in \Z} p(x_i,y)\bigg(K -  \sum_{j=1}^n 1(y=x_j)\bigg) ( f(\bx^{x_i,y}) - f(\bx) ), 
\]
where $\bx = (x_1, \ldots, x_n)$ and $\bx^{x_i,y} = (x_1, \ldots, x_{i-1},y,x_{i+1}, \ldots, x_n)$. 
Note in particular that 
\[
	\Omega_n^1 = \{\bx \in \Z^n : x_j \ne x_k\;\text{for}\; j \ne k\} = \Z^{(n)}, 
\]
using the notation defined in \eqref{eq:(n)notation}.

Let $\{V_n^K(t) : t \ge 0\}$ denote the Markov semigroup corresponding to $\sV_n^K$, and let $\{U_n^K(t) : t \ge 0\}$ denote the semigroup of independent $n$-particle motion with jumps at rate $K$, which has state space $\Z^n$ and corresponding generator 
\[
	\sU_n^K f(\bx) = K \sum_{i=1}^n \sum_{y \in \Z}p(x_i,y) ( f(\bx^{x_i,y}) - f(\bx) ). 
\]
Note that $\sU_n^K = K\sU_n^1$ and $U_n^K(t) = U_n^1(Kt)$. 

For the next lemma, a function $f: \Z^2 \to \R$ is said to be positive definite if 
\[
	\sum_{x,y \in \Z} \beta(x)f(x,y)\beta(y) \ge 0 \quad\mbox{whenever}\quad \sum_{x \in \Z} |\beta(x)| < \infty \quad \mbox{and} \quad \sum_{x \in \Z} \beta(x) = 0. 
\]
A function $f : \Z^n \to \R$ is positive definite if it is positive definite in each pair of variables. 
The following semigroup inequality is derived in \cite{GiaRedVaf2010}.

\begin{lemma}[\cite{GiaRedVaf2010}]\label{VUineq} For any symmetric, positive definite $f : \Omega_n^K \to \R$ and all $t \ge 0$, 
\[
	V_n^K(t) f \le U_n^K(t) f. 
\]
\end{lemma}

This in turn implies the following corollary, which is based on the observation that, as long as $f$ is a {\em symmetric} function, 
\begin{equation}\label{eq:partpos2stir}
\begin{aligned}
	\E[ f(\xi_t^{x_1j_1}, \ldots, \xi_t^{x_nj_n})] &= E[f(Y_t^1, \ldots, Y_t^n) | Y_0^1 = x_1, \ldots, Y_0^n = x_n] \\
	&= V_n^K(t) f(x_1, \ldots, x_n), 
\end{aligned}
\end{equation}
for any $((x_1,j_1), \ldots, (x_n,j_n)) \in (\Z \times [K])^{(n)}$. 
This is because in an $n$-particle system, while each stirring variable does not track a single particle, the {\em set} of particle positions and the {\em set} of stirring positions are the same. Used implicitly in \eqref{eq:partpos2stir} is the observation that $((x_1,j_1), \ldots, (x_n,j_n)) \in (\Z \times [K])^{(n)}$ implies $(x_1, \ldots, x_n) \in \Omega_n^K$.

For $K=1$, the following is given in \cite[Ch. VIII]{LigBook05}. 

\begin{corollary}\label{negassoc} Let $A \subset \Z$, and let $((x_1, j_1), \ldots, (x_n,j_n)) \in (\Z \times [K])^{(n)}$. Then for any $t \ge 0$, 
\[
	\P \bigg( \bigcap_{k=1}^n \{\xi_{t/K}^{x_kj_k} \in A\} \bigg) \le \prod_{k=1}^n P_{x_k}(\zeta_{t} \in A).  
\]
\end{corollary}

\begin{proof} Apply Lemma \ref{VUineq} to the symmetric, positive definite function $1_{A^n}$, 
noting \eqref{eq:partpos2stir} and $U_n^K(\frac t K) = U_n^1(t)$. 
\end{proof}

At several points (including in the proof of the following lemma), we will apply the fact that, by definition, a symmetic operator $\sO$ on a Hilbert space satisfies $\il \sO u, v \ir = \il u, \sO v \ir$ for any $u,v$ in its domain. In particular, $U_n^K(t)$ and $V_n^K(t)$ are symmetric operators, so that 
\begin{equation}\label{eq:Usymmop}
	\il f, U_n^K(t)g \ir := \sum_{\bx \in \Z^n} f(\bx) U_n^K(t) g(\bx) = \sum_{\bx \in \Z^n} g(\bx) U_n^K(t) f(\bx), 
\end{equation}
and 
\begin{equation}\label{eq:Vsymmop}
	\sum_{\bx \in \Omega_n^K} f(\bx) V_n^K(t) g(\bx) = \sum_{\bx \in \Omega_n^K} g(\bx) V_n^K(t) f(\bx), 
\end{equation}
for any bounded $f$ and $g$. 
Note that, with this inner product notation, we may re-express the intensity measure in 
\eqref{eq:meanrep} as
\begin{equation}\label{eq:ENasip}
	\mu_t^\nu(A) = \sum_{x \in \Z} E_\nu[\eta(x)]P_x(\zeta_{Kt} \in A) = \il E_\nu[\eta(\cdot)], U_1^1(Kt) 1_{A} \ir, \qquad A \in \sB_\R. 
\end{equation}

We now present several lemmas involving comparisons of $U_n^K(t)$ and $V_n^K(t)$ that are used in Section \ref{otherproofs}.

\begin{lemma}\label{sumsymmetry} 
Suppose that $f_1, f_2, g: \Z^n \to \R$ are nonnegative and bounded, where (i) $f_1$ is symmetric and positive definite, (ii) $f_2 \le f_1$ on $\Omega_n^K$, and (iii) $f_2 = f_1$ on $\Omega_n^1 = \Z^{(n)}$. Then for any $t \ge 0$, 
\begin{align*}
	0&\le \sum_{\bx \in \Z^n} f_1(\bx) U_n^K(t) g(\bx) - \sum_{\bx \in \Omega_n^K} f_2(\bx) V_n^K(t) g(\bx) \\
	&\le \sum_{\bx \in \Omega^1_n} g(\bx) \big[U_n^K(t) - V_n^K(t)\big] f_1(\bx)  + \sum_{\bx \in (\Omega_n^1)^c} g(\bx) U_n^K(t) f_1(\bx),  
\end{align*}
provided that all sums converge.
\end{lemma}

\begin{proof} 
For nonnegativity, first use \eqref{eq:Usymmop}, \eqref{eq:Vsymmop}, $f_2 \le f_1$ on $\Omega_n^K$, and $g \ge 0$ to obtain
\begin{align} \nonumber
	&\sum_{\bx \in \Z^n} f_1(\bx) U_n^K(t) g(\bx) - \sum_{\bx \in \Omega_n^K} f_2(\bx) V_n^K(t) g(\bx) \\ \nonumber
	&\ge \sum_{\bx \in \Z^n} f_1(\bx) U_n^K(t) g(\bx) - \sum_{\bx \in \Omega_n^K} f_1(\bx) V_n^K(t) g(\bx)\\ \label{eq:nonnegativething}
	&= \sum_{\bx \in \Omega_n^K} g(\bx) \big[ U_n^K(t) - V_n^K(t) \big] f_1(\bx) + \sum_{\bx \in (\Omega_n^K)^c} g(\bx) U_n^K(t) f_1(\bx). 
\end{align}
Because $f_1, g \ge 0$, $g(\bx) U_n^K(t) f_1(\bx) \ge 0$ for all $\bx$. From Lemma \ref{VUineq} and because $f_1$ is symmetric and positive definite, $g(\bx) \big[ U_n^K(t) - V_n^K(t) \big] f_1(\bx) \ge 0$ for $\bx \in \Omega_n^K$. Thus the quantity \eqref{eq:nonnegativething} is nonnegative, which establishes the first inequality of Lemma \ref{sumsymmetry}. 

For the second inequality, since $\sum_{\bx \in \Omega_n^K \setminus \Omega_n^1} f_2(\bx) V_n^K(t)g(\bx) \ge 0$ and $f_1 = f_2$ on $\Omega_n^1$, 
\begin{align*}
	&\sum_{\bx \in \Z^n} f_1(\bx) U_n^K(t) g(\bx) - \sum_{\bx \in \Omega_n^K} f_2(\bx) V_n^K(t) g(\bx) \\ 
	&\le \sum_{\bx \in \Z^n} f_1(\bx) U_n^K(t) g(\bx) - \sum_{\bx \in \Omega_n^1} f_1(\bx) V_n^K(t) g(\bx) \\ 
	&= \sum_{\bx \in \Omega^1_n} g(\bx) \big[U_n^K(t) - V_n^K(t)\big] f_1(\bx)  + \sum_{\bx \in (\Omega_n^1)^c} g(\bx) U_n^K(t) f_1(\bx). \qedhere
\end{align*}
\end{proof}

\begin{lemma}\label{SGdiffform} 
For any $A \subset \Z$ and $\bx \in \Omega_n^K$, 
\begin{align*}
	&\left[U_n^K(\tfrac{t}{K}) - V_n^K(\tfrac{t}{K})\right] 1_{A^n}(\bx) \\
	&= \int_0^t V_n^K(\tfrac{t-s}{K}) \sum_{\{i,j\} \subset [n]} p(x_i,x_j) 
	( P_{x_i}(\zeta_s \in A)- P_{x_j}(\zeta_s \in A) )^2\prod_{k\in[n]\setminus\{i,j\}} P_{x_k}(\zeta_s \in A)\,ds.  
\end{align*}
\end{lemma}

\begin{proof}
The integration by parts formula \cite[Ch. VIII]{LigBook05} gives 
\begin{align}\nonumber 
	\left[U_n^K(\tfrac{t}{K}) - V_n^K(\tfrac{t}{K})\right] 1_{A^n} &= \int_0^{\tfrac{t}{K}} V_n^K(\tfrac{t}{K}-s) \left(\sU_n^K - \sV_n^K\right)U_n^K(s) 1_{A^n}\,ds \\ \label{eq:intparts}
	&= \frac 1 K \int_0^{t} V_n^K(\tfrac{t-s}{K}) \left(\sU_n^K - \sV_n^K\right)U_n^K(\tfrac{s}{K}) 1_{A^n}\,ds.
\end{align}
Note that 
\begin{align*}
	\left(\sU_n^K - \sV_n^K\right)f(\bx) &= K\sum_{i=1}^n\sum_{j=1}^n p(x_i, x_j)( f(\bx^{x_i,x_j}) - f(\bx) ) \\ 
	&= \frac{K}{2} \sum_{i=1}^n\sum_{j=1}^n p(x_i, x_j) ( f(\bx^{x_i,x_j}) + f(\bx^{x_j,x_i}) - 2f(\bx) ), 
\end{align*}
where the last line is obtained by switching the rolls of $x_i$ and $x_j$ in the sum of the first line and then adding it to itself, using $p(x_i, x_j) = p(x_j, x_i)$. Then, 
\begin{align*}
	&K^{-1}\left(\sU_n^K - \sV_n^K\right)U_n^K(\tfrac{s}{K}) 1_{A^n}(\bx) \\
	&= \sum_{\{i,j\}\subset [n]} p(x_i,x_j) \big( \big[U_n^K(\tfrac{s}{K})1_{A^n}\big](\bx^{x_i,x_j}) + \big[U_n^K(\tfrac{s}{K})1_{A^n}\big](\bx^{x_j,x_i})  - 2\big[U_n^K(\tfrac{s}{K})1_{A^n}\big](\bx) \big) \\
	&= \sum_{\{i,j\}\subset [n]} p(x_i,x_j) \bigg( P_{x_i}(\zeta_s \in A)^2 \prod_{k \in [n] \setminus \{i,j\}} P_{x_k}(\zeta_s \in A) + P_{x_j}(\zeta_s \in A)^2 \prod_{k \in [n] \setminus \{i,j\}} P_{x_k}(\zeta_s \in A) \\
	&\qquad\qquad\qquad\qquad\qquad - \prod_{k=1}^n P_{x_k}(\zeta_s \in A) \bigg) \\
	&= \sum_{\{i,j\} \subset [n]} p(x_i,x_j) ( P_{x_i}(\zeta_s \in A)- P_{x_j}(\zeta_s \in A) )^2\prod_{k\in[n]\setminus\{i,j\}} P_{x_k}(\zeta_s \in A). 
\end{align*}
Now plug this into \eqref{eq:intparts}.  
\end{proof}

For use in the proof of the next result, we give a lemma which follows from a straight-forward computation:

\begin{lemma}\label{nto2} Suppose that $\alpha : \Z^2 \to \R$ is symmetric, and let $\beta : \Z \to \R$. Then for $n \ge 2$, 
\[
	\sum_{\bx \in \Z^n} \sum_{\{i,j\} \subset [n]} \alpha(x_i,x_j) \prod_{k \in [n] \setminus \{i,j\}} \beta(x_k) = {n \choose 2} \bigg( \sum_{\bx \in \Z^{(2)}} \alpha(x_1,x_2) \bigg) \bigg( \sum_{x \in \Z} \beta(x) \bigg)^{n-2}. 
\]
\end{lemma}

Next, we have a bound on quantities like those in Lemma \ref{sumsymmetry}. Recall the inner product given in \eqref{eq:Usymmop}, and note that when $f = 1_A$ and $g = 1_B$ for $A, B \subset \Z$, 
\begin{equation}\label{eq:setip}
	\il 1_A, U_1^1(t)1_B \ir = \sum_{x \in A} P_x(\zeta_t \in B). 
\end{equation}
Also recall the quantities $\kappa_t(A,B)$ and $\tau_t(A,B)$ from Definition \ref{kappataudef}. 

\begin{lemma}\label{kappataubounds} Let $A, B_1, \ldots, B_n \subset \Z$, and suppose that $B \supset \bigcup_{k=1}^n B_k$. Then, 
\[
	\sum_{\bx \in \Omega_n^1} 1_{B_1 \times \cdots \times B_n}(\bx) \big[ U_n^K(\tfrac{t}{K}) - V_n^K(\tfrac{t}{K}) \big] 1_{A^n}(\bx) \le {n \choose 2} \big( \il 1_B, U_1^1(t)1_A \ir \big)^{n-2} \kappa_t(B, A), 
\]
and 
\[
	\sum_{\bx \in (\Omega_n^1)^c} 1_{B_1 \times \cdots \times B_n}(\bx) U_n^K(\tfrac{t}{K}) 1_{A^n}(\bx) \le {n \choose 2} \big( \il 1_B, U_1^1(t)1_A \ir \big)^{n-2} \tau_t(B, A). 
\]
\end{lemma}

\begin{proof} Since $1_{A^n}$ is symmetric and positive definite, we have from Lemma \ref{VUineq}, $B_k \subset B$ for all $k$, and $\Omega_n^1 \subset \Omega_n^K$ that 
\[
	\sum_{\bx \in \Omega_n^1} 1_{B_1 \times \cdots \times B_n}(\bx) \big[ U_n^K(\tfrac{t}{K}) - V_n^K(\tfrac{t}{K}) \big] 1_{A^n}(\bx) \le \sum_{\bx \in \Omega_n^K} 1_{B^n}(\bx) \big[ U_n^K(\tfrac{t}{K}) - V_n^K(\tfrac{t}{K}) \big] 1_{A^n}(\bx). 
\] 
From Lemma \ref{SGdiffform} and \eqref{eq:Vsymmop}, 
\begin{align*}
	&\sum_{\bx \in \Omega_n^K} 1_{B^n}(\bx) \big[ U_n^K(\tfrac{t}{K}) - V_n^K(\tfrac{t}{K}) \big] 1_{A^n}(\bx) \\
	&= \int_0^t \sum_{\bx \in \Omega_n^K} 1_{B^n}(\bx) V_n^K(\tfrac{t-s}{K}) \sum_{\{i,j\} \subset [n]} p(x_i,x_j) \big( P_{x_i}(\zeta_s \in A) - P_{x_j}(\zeta_s \in A) \big)^2 \prod_{k \ne i,j} P_{x_k}(\zeta_s \in A) \,ds \\
	&= \int_0^t \sum_{\bx \in \Omega_n^K} \sum_{\{i,j\} \subset [n]} p(x_i,x_j) \big( P_{x_i}(\zeta_s \in A) - P_{x_j}(\zeta_s \in A) \big)^2 \prod_{k \ne i,j} P_{x_k}(\zeta_s \in A) V_n^K(\tfrac{t-s}{K}) 1_{B^n}(\bx)  \,ds. 
\end{align*}
Since $1_{B^n}$ is symmetric and positive definite, Lemma \ref{VUineq} implies  
\[
	V_n^K(\tfrac{t-s}{K}) 1_{B^n}(\bx) \le U_n^K(\tfrac{t-s}{K}) 1_{B^n}(\bx) = \prod_{k=1}^n P_{x_k}(\zeta_{t-s} \in B). 
\]
By combining the previous three displays and using Lemma \ref{nto2}, we obtain 
\begin{align*}
	&\sum_{\bx \in \Omega_n^1} 1_{B_1 \times \cdots \times B_n}(\bx) \big[ U_n^K(\tfrac{t}{K}) - V_n^K(\tfrac{t}{K}) \big] 1_{A^n}(\bx) \\
	&\le  \int_0^t \sum_{\bx \in \Omega_n^K} \sum_{\{i,j\} \subset [n]} p(x_i,x_j) \big( P_{x_i}(\zeta_s \in A) - P_{x_j}(\zeta_s \in A) \big)^2P_{x_i}(\zeta_{t-s} \in B)P_{x_j}(\zeta_{t-s} \in B) \\
	&\qquad\qquad\qquad\qquad \times \prod_{k \ne i,j} P_{x_k}(\zeta_s \in A)P_{x_k}(\zeta_{t-s} \in B) \,ds \\
	&\le  \int_0^t \sum_{\bx \in \Z^n} \sum_{\{i,j\} \subset [n]} p(x_i,x_j) \big( P_{x_i}(\zeta_s \in A) - P_{x_j}(\zeta_s \in A) \big)^2P_{x_i}(\zeta_{t-s} \in B)P_{x_j}(\zeta_{t-s} \in B) \\
	&\qquad\qquad\qquad\qquad \times \prod_{k \ne i,j} P_{x_k}(\zeta_s \in A)P_{x_k}(\zeta_{t-s} \in B) \,ds \\
	&= \int_0^t {n \choose 2} \bigg( \sum_{x, y \in \Z} p(x,y) \big( P_{x}(\zeta_s \in A) - P_{y}(\zeta_s \in A) \big)^2P_{x}(\zeta_{t-s} \in B)P_{y}(\zeta_{t-s} \in B) \bigg) \\
	&\qquad\qquad\qquad\qquad \times \bigg( \sum_{x \in \Z} P_x(\zeta_s \in A)P_x(\zeta_{t-s} \in B) \bigg)^{n-2}\,ds. 
\end{align*}
Now note that, for any $0 \le s \le t$, \eqref{eq:Usymmop} gives 
\begin{align*}
	 \sum_{x \in \Z} P_x(\zeta_s \in A)P_x(\zeta_{t-s} \in B) &= \sum_{x \in \Z} U_1^1(s)1_A(x)U_1^1(t-s)1_B(x) \\
	 &= \il U_1^1(s)1_A, U_1^1(t-s)1_B \ir \\
	 &= \il U_1^1(t-s)U_1^1(s)1_A, 1_B \ir = \il U_1^1(t)1_A, 1_B \ir = \il 1_B, U_1^1(t)1_A \ir. 
\end{align*}
Therefore, 
\begin{align*}
	&\sum_{\bx \in \Omega_n^1} 1_{B_1 \times \cdots \times B_n}(\bx) \big[ U_n^K(\tfrac{t}{K}) - V_n^K(\tfrac{t}{K}) \big] 1_{A^n}(\bx) \\
	&\le {n \choose 2} \big( \il 1_B, U_1^1(t)1_A \ir \big)^{n-2} \int_0^t \sum_{x, y \in \Z} \big( P_{x}(\zeta_s \in A) - P_{y}(\zeta_s \in A) \big)^2P_{x}(\zeta_{t-s} \in B)P_{y}(\zeta_{t-s} \in B)\,ds \\
	&= {n \choose 2} \big( \il 1_B, U_1^1(t)1_A \ir \big)^{n-2} \kappa_t(B,A). 
\end{align*}

Next, since $\bx \in (\Omega_n^1)^c$ implies that $x_i = x_j$ for at least one pair $\{i, j\} \subset [n]$, 
\begin{align*}
	&\sum_{\bx \in (\Omega_n^1)^c} 1_{B_1 \times \cdots \times B_n}(\bx) U_n^K(\tfrac{t}{K}) 1_{A^n}(\bx) \le \sum_{\bx \in (\Omega_n^1)^c} 1_{B^n}(\bx) U_n^K(\tfrac{t}{K}) 1_{A^n}(\bx) \\
	&\le \sum_{\{i, j \} \subset [n]} \underset{x_i = x_j}{\sum_{\bx \in B^n}} \prod_{k=1}^n P_{x_k}(\zeta_t \in A) \\
	&= {n \choose 2} \underset{x_1 = x_2}{\sum_{\bx \in B^n}} \prod_{k=1}^n P_{x_k}(\zeta_t \in A) 
	= {n \choose 2} \bigg( \sum_{x \in B} P_{x}(\zeta_t \in A)^2 \bigg) \bigg( \sum_{x \in B} P_x(\zeta_t \in A) \bigg)^{n-2}. 
\end{align*}
Recalling \eqref{eq:setip} and $\tau_t(B, A) = \sum_{x \in B}P_x(\zeta_t \in A)^2$, the above shows 
\[
	\sum_{\bx \in (\Omega_n^1)^c} 1_{B_1 \times \cdots \times B_n}(\bx) U_n^K(\tfrac{t}{K}) 1_{A^n}(\bx) \le {n \choose 2} \big( \il 1_B, U_1^1(t)1_A \ir \big)^{n-2} \tau_t(B, A). \qedhere
\]
\end{proof}

Combining Lemmas \ref{sumsymmetry} and \ref{kappataubounds}, we obtain the following. 

\begin{corollary}\label{maincor} Let $A, B_1, \ldots, B_n \subset \Z$, and suppose that $B \supset \bigcup_{k=1}^n B_k$. Then for any function $h$ on $\Omega_n^K$ with $0 \le h \le 1$ and $h \equiv 1$ on $\Omega_n^1$, 
\begin{align*}
	0 &\le \sum_{\bx \in A^n} U_n^K(\tfrac{t}{K})1_{B_1 \times \cdots \times B_n}(\bx) - \sum_{\bx \in \Omega_n^K \cap A^n} h(\bx) V_n^K(\tfrac{t}{K})1_{B_1 \times \cdots \times B_n}(\bx) \\
	&\le {n \choose 2} \big(\il 1_B, U_1^1(t)1_A \ir \big)^{n-2} ( \kappa_t(B,A) + \tau_t(B,A) ). 
\end{align*}
\end{corollary}

\begin{proof} Let $f_1 = 1_{A^n}$, $f_2 = h1_{A^n}$ and $g = 1_{B_1 \times \cdots \times B_n}$ in Lemma \ref{sumsymmetry}. Apply Lemma \ref{kappataubounds} to the resulting upper bound. 
\end{proof}

\section{Proofs of Propositions \ref{meanconv}, \ref{factmomentbound}, \ref{kappatauconv}, and \ref{genmzrbound}}\label{otherproofs}

Here we prove the results stated in the proof outline of Section \ref{outline}. Namely, we prove Propositions \ref{meanconv}, \ref{factmomentbound}, \ref{genmzrbound}, and \ref{kappatauconv}, in that order.

\subsection{Proof of Proposition \ref{meanconv}}

We recall that intensity measure $\mu^\eta_t$ in \eqref{eq:meanrep}, with $\mu_t^L = \mu_t^{\eta^L}$ for the deterministic profile $\eta^L$ given in \eqref{eq:fullstep}. 
The following is Theorem 5.2 in \cite{ConSet2023}. 

\begin{lemma}[\cite{ConSet2023}]\label{K=1Lmean} Fix $K = 1$ and $x \in \R$. 
\begin{enumerate}[(a)]

\item If $L\big(\frac{\log t}{t}\big)^{1/2} \to \psi \in (0,\infty]$ as $t \to \infty$, then 
\[
	\lim_{t\to\infty} 
	\mu_t^L \circ v_t^{-1}(x, \infty) = \sigma(1 - e^{-\psi/\sigma})e^{-x}. 
\]

\item If $L\big(\frac{\log t}{t}\big)^{1/2} \to 0$ as $t \to \infty$, then 
\[
	\lim_{t\to\infty} 
	\mu_t^L \circ v_{t,L}^{-1}(x, \infty) = e^{-x}. 
\]

\end{enumerate}
\end{lemma}

Note that, for general $K \ge 1$ and $A \in \sB(\R)$, 
\begin{equation}\label{eq:mean1toK}
\begin{aligned}
	\E_L[N_{t/K}(A)] &= \sum_{x \in \Z}  \eta_L(x)P_x( \zeta_t \in A ) \\
	&= K \sum_{-L < x \le 0} P_x(\zeta_t \in A) = K \E_L^1[N_t(A)], 
\end{aligned}
\end{equation}
where $\E_L^1$ indicates expectation with initial condition $\eta_L$ and $K = 1$ particles allowed per site. Then we have the following corollary in the $K \ge 1$ setting: 

\begin{corollary}\label{Kcor}
Let $K \ge 1$ and $A = (a,b]$ for $-\infty < a < b < \infty$. 
\begin{enumerate}[(a)]

\item If $L\big(\frac{\log t}{t}\big)^{1/2} \to \psi \in (0,\infty]$ as $t \to \infty$, then 
\[
	\lim_{t\to\infty} 
	\mu_{t/K}^L \circ v_t^{-1}(A) = K\sigma(1 - e^{-\psi/\sigma})\lambda(A). 
\]

\item If $L\big(\frac{\log t}{t}\big)^{1/2} \to 0$ as $t \to \infty$, then 
\[
	\lim_{t\to\infty} 
	\mu_{t/K}^L \circ v_{t,L}^{-1}(A) = K\lambda(A). 
\]

\end{enumerate}
\end{corollary}

\begin{proof} Consider part (a). From Lemma \ref{K=1Lmean} and \eqref{eq:mean1toK}, 
\begin{align*}
	\mu_{t/K}^L \circ v_t^{-1}(A) &= \E_L[N_{t/K} \circ z_t^{-1}(a,\infty)] - \E_L[N_{t/K} \circ v_{t}^{-1}(b,\infty)] \\
	&= K\big( \E_L^1[N_{t} \circ v_t^{-1}(a,\infty)] - \E_L^1[N_{t} \circ v_t^{-1}(b,\infty)] \big)\\
	&\to K\sigma(1 - e^{-\psi/\sigma}) ( e^{-a} - e^{-b} ) = K\sigma(1 - e^{-\psi/\sigma})\lambda(A). 
\end{align*}
Part (b) follows similarly. 
\end{proof}

We give the proof of Proposition \ref{meanconv} after the next lemma. Recall that $f(t) \sim g(t)$ as $t \to \infty$ denotes $\lim_{t\to\infty} f(t)/g(t) = 1$, and recall the truncation $\nu_L$ of $\nu \in \sP_K$ from \eqref{eq:Lmeasure}. 

\begin{lemma} Let $A = (a,b]$ for $-\infty < a < b < \infty$. Suppose $L \uparrow \infty$ as $t \to \infty$, and let $w_t$ denote either $v_t$ or $v_{t,L}$ for all $t$. If $\nu \in \sP_K^{\text{step}}$ satisfies Condition \ref{initialprofilecond}, then  
\[
	\mu_{t/K}^{\nu_L} \circ w_t^{-1}(A)  \sim (c_\nu/K) \mu_{t/K}^{L} \circ w_t^{-1}(A) , \qquad t \to \infty.
\]
\end{lemma}

\begin{proof} Because $\mu_{t/K}^\nu \circ w_t^{-1}(a,b] = \mu_{t/K}^\nu \circ w_t^{-1}(a,\infty) - \mu_{t/K}^\nu \circ w_t^{-1}(b, \infty)$, 
 it suffices to prove the result for $A = (a, \infty)$. 
For $\eps > 0$, let $m$ be large enough that 
\[
	\bigg|\frac{1}{k} \sum_{j=0}^{k-1} E_\nu[\eta(-j)] - c_\nu\bigg| < \eps, \qquad k \ge m, 
\]
which is possible under Condition \ref{initialprofilecond}. 
Recall that $c_\nu > 0$. 

Note that $w_t^{-1}(a, \infty) = (w_t^{-1}(a), \infty)$ since $w_t$ is a one-to-one nondecreasing function. 
We simplify notation by setting $Z_t = (\zeta_t - \lfloor w_t^{-1}(a) \rfloor)_+$.  
Then we have  
\begin{align*}
	\mu_{t/K}^{\nu_L} \circ w_t^{-1}(a, \infty) &= \sum_{x \in \Z} E_{\nu_L}[\eta(x)]P_x(\zeta_t \in w_t^{-1}(a, \infty)) \\
	&= \sum_{-L < x \le 0} E_{\nu}[\eta(x)]P_x(\zeta_t > w_t^{-1}(a)) \\
	&= \sum_{0 \le x < L} E_\nu[\eta(-x)] P_0( Z_t > x ) = E_0 \bigg[ \sum_{x = 0}^{L \wedge Z_t - 1} E_\nu[\eta(-x)] \bigg],  
\end{align*}
with the convention that $\sum_{x=0}^{-1}E_\nu[\eta(-x)] = 0$. 
In particular, 
\[
	\mu_{t/K}^L \circ w_t^{-1}(a, \infty) = K E_0[  L \wedge Z_t  ]. 
\]

Now suppose $t$ is sufficiently large that $L > m$. Then,  
\begin{align*}
	&E_0 \bigg[ \sum_{x = 0}^{L \wedge Z_t - 1} E_\nu[\eta(-x)] \bigg] \\
	&= E_0\bigg[ (L \wedge Z_t) 1(Z_t > m) \cdot \frac{1}{ L \wedge Z_t}  \sum_{x = 0}^{L \wedge Z_t - 1} E_\nu[\eta(-x)] \bigg] 
	+ E_0\bigg[ 1(0 < Z_t \le m) \sum_{x = 0}^{L \wedge Z_t - 1} E_\nu[\eta(-x)] \bigg] \\
	&\le (c_\nu + \eps) E_0[ L \wedge Z_t ] + KE_0[(L \wedge Z_t)1(0 < Z_t \le m)] \\
	&\le (c_\nu + \eps) E_0[ L \wedge Z_t ] + Km P_0(Z_t > 0) = \frac{c_\nu + \eps}{K} \mu_{t/K}^L \circ w_t^{-1}(a, \infty) + Km P_0(\zeta_t > w_t^{-1}(a)). 
\end{align*}
From \eqref{eq:zfirstorder}, \eqref{eq:zLfirstorder}, and the central limit theorem, $P_0(\zeta_t > w_t^{-1}(a)) \to 0$ as $t \to \infty$. 
Moreover, $\liminf_{t\to\infty} \mu_{t/K}^L \circ w_t^{-1}(a, \infty) > 0$ from Corollary \ref{Kcor}. Therefore,  
\begin{equation}\label{eq:limsup}
\begin{aligned}
	\limsup_{t\to\infty} \frac{K \mu_{t/K}^{\nu_L} \circ w_t^{-1}(a, \infty)}{c_\nu \mu_{t/K}^L \circ w_t^{-1}(a, \infty)} &\le \limsup_{t\to\infty} \bigg( 1 + \frac{\eps}{c_\nu} + \frac{K^2mP_0(\zeta_t > w_t^{-1}(a)) }{c_\nu \mu_{t/K}^L \circ w_t^{-1}(a, \infty)} \bigg) 
	= 1 + \frac{\eps}{c_\nu}.  
\end{aligned}
\end{equation}
On the other hand, 
\begin{align*}
	\mu_{t/K}^{\nu_L} \circ w_t^{-1}(a, \infty) &\ge E_0\bigg[ (L \wedge Z_t) 1(Z_t > m) \cdot \frac{1}{ L \wedge Z_t}  \sum_{x = 0}^{L \wedge Z_t - 1} E_\nu[\eta(-x)] \bigg]  \\
	&\ge (c_\nu - \eps) E_0[(L \wedge Z_t)1(Z_t > m)] \\
	&\ge (c_\nu - \eps) E_0[(L \wedge Z_t)] - c_\nu E_0[(L \wedge Z_t)1(Z_t \le m)] \\
	&\ge (c_\nu - \eps) E_0[(L \wedge Z_t)] - c_\nu m P_0(Z_t > 0) \\
	&=  \frac{c_\nu - \eps}{K} \mu_{t/K}^L \circ w_t^{-1}(a, \infty) - c_\nu m P_0(\zeta_t > w_t^{-1}(a)). 
\end{align*}
Then, 
\begin{equation}\label{eq:liminf}
	\liminf_{t\to\infty} \frac{K \mu_{t/K}^{\nu_L} \circ w_t^{-1}(a, \infty)}{c_\nu \mu_{t/K}^L \circ w_t^{-1}(a, \infty)} \ge 1 - \frac{\eps}{c_\nu}. 
\end{equation}
Letting $\eps \downarrow 0$ in \eqref{eq:limsup} and \eqref{eq:liminf} completes the proof. 
\end{proof}

\begin{proof}[\bf Proof of Proposition \ref{meanconv}] Let $A = (a,b]$. The previous lemma establishes \eqref{eq:nuLvsLmean}. Consider part (a), in which $L(\frac{\log t}{t})^{1/2} \to \psi \in (0, \infty]$. From \eqref{eq:nuLvsLmean} and Corollary \ref{Kcor} (a), 
\[
	\lim_{t\to\infty} \mu^{\nu_L}_{t/K} \circ v_t^{-1}(A) = \frac{c_\nu}{K} \lim_{t\to\infty} \mu^{L}_{t/K} \circ v_t^{-1}(A) = c_\nu\sigma(1 - e^{-\psi/\sigma})\lambda(A). 
\]
Part (b) of Proposition \ref{meanconv} follows similarly, using \eqref{eq:nuLvsLmean} and Corollary \ref{Kcor} (b). 
\end{proof}

\subsection{Proof of Proposition \ref{factmomentbound}}

Here we prove Proposition \ref{factmomentbound}, beginning with the following lemma. 
Throughout this section, $A_k = (a_k, b_k]$, $1 \le k \le m$, are disjoint and $\eta \in \{0, K\}^\Z$ is a fixed, deterministic initial profile. 
Set $M = n_1 + \cdots + n_m$.

Define $F : \Z^M \to [0,1]$ as follows. Let 
\[
	M_0 = 0, \qquad M_k = n_1 + \cdots + n_k \quad \text{for} \quad k = 1, \ldots, m. 
\]
(Note that $M_m = M$). Then let 
\begin{equation}\label{eq:FMdef}
	F(\bx) = \prod_{k=1}^m \; \prod_{r = M_{k-1} + 1}^{M_k} 1(x_r \in A_k). 
\end{equation}

Recall that $\Omega_M^K$ denotes the state space of an $M$-particle $K$-SEP system, defined in \eqref{eq:Omegadef}. 
Recalling also that $\eta(x) \in \{0,K\}$ for all $x \in \Z$, set $H = \{\eta > 0\} = \{\eta = K\}$ and $H_K = H \times [K]$. 
Using the notation in \eqref{eq:(n)notation}, define $h: \Omega_M^K \to \R$ by 
\begin{equation}\label{eq:tildehdef}
	h(\bx) = \big| \{ \bj \in [K]^M : ((x_1, j_1), \ldots, (x_M, j_M)) \in H_K^{(M)} \} \big|, 
\end{equation}
where $|\cdot|$ indicates cardinality. 

The system with initial condition $\eta$ can be described at time $t \ge 0$ by the collection $\{\xi_t^{xj} : x \in H, j = 1, \ldots, K\}$ (cf. \eqref{eq:etadefasstir}), and there is a one-to-one correspondence between $M$ starting positions $x_1, \ldots, x_M \in H$ and distinct stirring variables $\xi_t^{x_1j_1}, \ldots, \xi_t^{x_Mj_M}$ if and only if $(x_1, \ldots, x_M) \in \Omega_M^K$. 
Then $h(x_1, \ldots, x_M)$ counts the number of possible labels $j_1, \ldots, j_M$ for the collection $\xi_t^{x_1j_1}, \ldots, \xi_t^{x_Mj_M}$. 
For example, if $K=1$ then $h(\bx) = 1(\bx \in H^{(M)})$ for $\bx \in \Omega_M^1$ and any value of $M$, whereas if $K \ge M$ then $h(\bx) = K^M$ for all $\bx$. 
We note for later use that, in general, 
\begin{equation}\label{eq:hinsum}
	\sum_{((x_1,j_1), \ldots, (x_M, j_M)) \in H_K^{(M)}} g(\bx) = \sum_{\bx \in \Omega_M^K \cap H^M} h(\bx)g(\bx), 
\end{equation}
for any bounded $g : \Omega_M^K \to \R$. 

The next lemma rewrites the quantity of interest in Proposition \ref{factmomentbound} in terms of the semigroup descriptions of the $M$-particle $K$-SEP and independent random walk processes as given in Section \ref{semigroup}. Recall the definition of the factorial measure in \eqref{eq:factmzrdef}. 

\begin{lemma}\label{semigroupform} For any $t \ge 0$, 
\begin{align*}
	&\prod_{k=1}^m ( \mu_t^\eta (A_k) )^{n_k} - \E_{\eta}\bigg[ \prod_{k=1}^m N_t^{(n_k)}(A_k^{n_k})  \bigg] = K^M\sum_{\bx \in H^M}U_M^K(t) F(\bx) - \sum_{\bx \in \Omega_M^K \cap H^M} h(\bx) V_M^K(t) F(\bx) .
\end{align*}
\end{lemma}

\begin{proof} Recall $H_K = \{\eta > 0\} \times [K]$. Then for any $n \in \N$ and $B \in \sB_\R$, 
\[
	N_t^{(n)}(B^n) = \sum_{((x_1, j_1), \ldots, (x_{n}, j_{n})) \in H_K^{(n)}} \; \prod_{r = 1}^n 1(\xi_t^{x_rj_r} \in B), \qquad \P_\eta\text{-surely.}
\]
Thus under $\P_\eta$ we have 
\begin{align*}
	\prod_{k=1}^m N_t^{(n_k)}(A_k^{n_k})  
	&= \prod_{k=1}^m \; \sum_{((x_{1},j_{1}),\ldots,(x_{n_k},j_{n_k})) \in H_K^{(n_k)}} \; \prod_{r=1}^{n_k} 1(\xi_t^{x_rj_r} \in A_k) \\
	&= \sum_{((x_1, j_1), \ldots, (x_M, j_M)) \in H_K^{(n_1)} \times \cdots \times H_K^{(n_m)}} \; \prod_{k=1}^m  \prod_{r=M_{k-1}+1}^{M_k} 1(\xi_t^{x_rj_r} \in A_k). 
\end{align*}
Because $A_1, \ldots, A_m$ are disjoint, the only nonzero terms in the sum in the above display are those for which $\bigcap_{k=1}^m \{(x_{M_{k-1}+1}, j_{M_{k-1} + 1}), \ldots, (x_{M_k}, j_{M_k})\} = \varnothing$. 
That is, 
\begin{align*}
	\prod_{k=1}^m N_t^{(n_k)}(A_k^{n_k})  
	&= \sum_{((x_1,j_1), \ldots, (x_M,j_M)) \in H_K^{(M)}} \; \prod_{k=1}^m  \prod_{r=M_{k-1}+1}^{M_k} 1(\xi_t^{x_rj_r} \in A_k) \\
	&= \sum_{((x_1,j_1), \ldots, (x_M,j_M)) \in H_K^{(M)}} F(\xi_t^{x_1j_1}, \ldots, \xi_t^{x_Mj_M}), \qquad \P_\eta\text{-surely,}
\end{align*}
and, taking expectation, 
\begin{equation} \label{eq:presym}
	\E_\eta\bigg[ \prod_{k=1}^m N_t^{(n_k)}(A_k^{n_k}) \bigg] = \sum_{((x_1,j_1), \ldots, (x_M,j_M)) \in H_K^{(M)}} \E[F(\xi_t^{x_1j_1}, \ldots, \xi_t^{x_Mj_M})]. 
\end{equation}

Note that $H_K^{(M)}$ is symmetric in the sense that it is invariant under permuting the indices of $((x_1, j_1), \ldots, (x_M, j_M))$. 
Then the sum in \eqref{eq:presym} is also invariant under permutations of the variable labels. 
Thus if we introduce a symmetrization of $F$ given by
\[
	\widehat F(x_1, \ldots, x_M) = \frac{1}{M!} \sum_{\sigma \in \mathbb{S}_M} F(x_{\sigma(1)}, \ldots, x_{\sigma(M)}), 
\]
where $\mathbb{S}_M$ is the symmetric group on $M$ elements, then by adding the right hand side of \eqref{eq:presym} to itself $M!$ times we obtain 
\begin{equation}\label{eq:postsym}
	\E_\eta\bigg[ \prod_{k=1}^m N_t^{(n_k)}(A_k^{n_k}) \bigg] =  \sum_{((x_1,j_1), \ldots, (x_M,j_M)) \in H_K^{(M)}} \E[\widehat F(\xi_t^{x_1j_1}, \ldots, \xi_t^{x_Mj_M})]. 
\end{equation}

Since $\widehat F$ is symmetric, from \eqref{eq:partpos2stir} we have 
\begin{align*}
	\E[\widehat F(\xi_t^{x_1j_1}, \ldots, \xi_t^{x_Mj_M})] = V_M^K(t) \widehat F(x_1, \ldots, x_M), 
\end{align*}
for any $((x_1,j_1), \ldots, (x_M,j_M)) \in H_K^{(M)}$. 
Substituting this identity into \eqref{eq:postsym} and then reversing the symmetrization procedure that lead to \eqref{eq:postsym} from \eqref{eq:presym}, using linearity of $V_M^K(t)$, we obtain 
\begin{align} \nonumber
	\E_\eta\bigg[ \prod_{k=1}^m N_t^{(n_k)}(A_k^{n_k}) \bigg] &= \sum_{((x_1,j_1), \ldots, (x_M,j_M)) \in H_K^{(M)}} V_M^K(t) \widehat F(x_1, \ldots, x_M) \\ \nonumber
	&= \sum_{((x_1,j_1), \ldots, (x_M,j_M)) \in H_K^{(M)}} \frac{1}{M!} \sum_{\sigma \in \mathbb{S}_M} V_M^K(t) F(x_1, \ldots, x_M) \\ \label{eq:nojs}
	&= \sum_{((x_1,j_1), \ldots, (x_M,j_M)) \in H_K^{(M)}} V_M^K(t) F(x_1, \ldots, x_M). 
\end{align}
In particular, each term of the sum in \eqref{eq:nojs} does not depend on $j_1, \ldots, j_M$. 
From \eqref{eq:hinsum}, \eqref{eq:nojs} becomes 
\begin{equation}\label{eq:factmomV}
	\E_\eta\bigg[ \prod_{k=1}^m N_t^{(n_k)}(A_k^{n_k}) \bigg] = \sum_{\bx \in \Omega_M^K \cap H^M} h(\bx)V_M^K(t) F(\bx). 
\end{equation}

Now we turn to the quantity $\prod_{k=1}^m(\mu_t^\eta(A_k))^{n_k}$. From \eqref{eq:meanrep}, and because $\eta(x) = K$ if and only if $x \in H$, 
\begin{equation}\label{eq:muwithH}
	\mu_t^\eta(A_k) = \sum_{x \in \Z} \eta(x) P_x(\zeta_{Kt} \in A_k) = K\sum_{x \in H} P_x(\zeta_{Kt} \in A_k), 
\end{equation}
for each $k$. 
Then, 
\begin{align*}
	\prod_{k=1}^m (\mu_t^\eta(A_k))^{n_k} &= \prod_{k=1}^m \bigg( K\sum_{x \in H} P_x(\zeta_{Kt} \in A_k) \bigg)^{n_k} \\
	&= \prod_{k=1}^m K^{n_k} \sum_{(x_1,\ldots, x_{n_k}) \in H^{n_k}} \prod_{r = 1}^{n_k} P_{x_r}(\zeta_{Kt} \in A_k) \\
	&=K^M \sum_{(x_1 \ldots, x_M) \in H^M} \prod_{k=1}^m \prod_{r=M_{k-1}+1}^{M_k} P_{x_r}(\zeta_{Kt} \in A_k) 
	= K^M\sum_{\bx \in H^M} U_M^K(t) F(\bx). 
\end{align*}
Combining the above display with \eqref{eq:factmomV} completes the proof. 
\end{proof}

Now we prove Proposition \ref{factmomentbound}. 

\begin{proof}[\bf Proof of Proposition \ref{factmomentbound}] From Lemma \ref{semigroupform}, 
\begin{equation}\label{eq:K^Mfactored}
\begin{aligned}
	&\prod_{k=1}^m ( \mu_{t/K}^\eta (A_k) )^{n_k} - \E_{\eta}\bigg[ \prod_{k=1}^m N_{t/K}^{(n_k)}(A_k^{n_k})  \bigg] \\
	&= K^M \bigg(\sum_{\bx \in H^M}U_M^K(\tfrac t K) F(\bx) - \sum_{\bx \in \Omega_M^K \cap H^M} K^{-M}h(\bx) V_M^K(\tfrac t K) F(\bx) \bigg).
\end{aligned}
\end{equation}
Note the following about the functions $F$ and $h$ defined in \eqref{eq:FMdef} and \eqref{eq:tildehdef}: 
\begin{enumerate}[(i)]

\item $F = 1_{B_1 \times \cdots \times B_M}$, where 
\[
	B_r = A_k \quad\text{for}\quad M_{k-1} + 1 \le r \le M_k, \quad 1 \le k \le m. 
\]
Moreover, recalling $A = (\min\{a_1, \ldots, a_m\}, \max\{b_1, \ldots, b_m\}]$, we have $\bigcup_{r=1}^M B_r \subset A$. 

\item $K^{-M}h(\bx) \le 1$ for all $\bx \in \Omega_M^K$. Moreover, if $\bx \in \Omega^1_M \cap H^M = H^{(M)}$ then for any $j_1, \ldots, j_M \in [K]$, $((x_1, j_1), \ldots, (x_M, j_M)) \in H_K^{(M)}$. So, $K^{-M}h(\bx) = 1$ on $\Omega_M^1$. 
\end{enumerate}
Using (i) and (ii), we now apply Corollary \ref{maincor} to \eqref{eq:K^Mfactored}. 

On the one hand, from the lower bound of Corollary \ref{maincor} the quantity in \eqref{eq:K^Mfactored} is nonnegative. 
On the other hand, the upper bound in Corollary \ref{maincor} yields
\begin{align*}
	&\prod_{k=1}^m ( \mu_{t/K}^\eta (A_k) )^{n_k} - \E_{\eta}\bigg[ \prod_{k=1}^m N_{t/K}^{(n_k)}(A_k^{n_k})  \bigg] \\
	&= K^M \bigg(\sum_{\bx \in H^M}U_M^K(\tfrac t K) 1_{B_1 \times \cdots \times B_M}(\bx) - \sum_{\bx \in \Omega_M^K \cap H^M} K^{-M}h(\bx) V_M^K(\tfrac t K) 1_{B_1 \times \cdots \times B_M}(\bx) \bigg) \\
	&\le K^M {M \choose 2} \big( \il 1_A, U_1^1(t) 1_H \ir \big)^{M-2} \big( \kappa_t(A, H) + \tau_t(A, H) \big). 
\end{align*}
Recall the inner product defined in \eqref{eq:Usymmop}. From \eqref{eq:muwithH} (cf. \eqref{eq:ENasip}) and the symmetry of $U_1^1(t)$, 
\begin{equation}\label{eq:iptomuinpf}
	K \il 1_A, U_1^1(t)1_H \ir = \il K1_H, U_1^1(t)1_A \ir = \il \eta, U_1^1(t)1_A \ir = \mu_{t/K}^\eta(A). 
\end{equation}
Therefore, 
\[
	\prod_{k=1}^m ( \mu_{t/K}^\eta (A_k) )^{n_k} - \E_{\eta}\bigg[ \prod_{k=1}^m N_{t/K}^{(n_k)}(A_k^{n_k})  \bigg] \le  K^2 {M \choose 2} \big( \mu_{t/K}^\eta(A) \big)^{M-2} \big( \kappa_t(A, H) + \tau_t(A, H) \big). \qedhere
\]
\end{proof}

\subsection{Proof of Proposition \ref{genmzrbound}}

In this section, we denote $\Z_K = \Z \times [K]$. First, note from \eqref{eq:meanrep} that for any probability measure $\nu$ on $\sX_K$, 
\begin{align*}
	(\mu_t^\nu(A))^n &= \bigg( \sum_{(x, j) \in \Z_K} \nu(\eta(x) \ge j)P_x(\zeta_{Kt} \in A) \bigg)^n \\
	&= \sum_{((x_1, j_1), \ldots, (x_n, j_n)) \in \Z_K^n} \prod_{k=1}^n \nu(\eta(x_k) \ge j_k)P_{x_k}(\zeta_{Kt} \in A). 
\end{align*}
Moreover, recalling the notations \eqref{eq:(n)notation} and \eqref{eq:factmzrdef}, 
\begin{align*}
	\E_\nu[N_t^{(n)}(A^n)] &= \E_\nu \bigg[ \sum_{((x_1, j_1), \ldots, (x_n, j_n)) \in \Z_K^{(n)}} \prod_{k=1}^n 1(\eta_0(x_k) \ge j_k, \xi_t^{x_kj_k} \in A) \bigg] \\
	&= \sum_{((x_1, j_1), \ldots, (x_n, j_n)) \in \Z_K^{(n)}} \nu\bigg( \bigcap_{k=1}^n \{\eta(x_k) \ge j_k\} \bigg) \P\bigg( \bigcap_{k=1}^n \{\xi_t^{x_kj_k} \in A\} \bigg). 
\end{align*}
Therefore, 
\begin{equation}\label{eq:meanpwrfactmomdiff}
\begin{aligned}
	(\mu^\nu_t(A))^n - \E_\nu[N_t^{(n)}(A^n)] &=  \sum_{((x_1, j_1), \ldots, (x_n, j_n)) \in \Z_K^n} \prod_{k=1}^n \nu(\eta(x_k) \ge j_k)P_{x_k}(\zeta_{Kt} \in A) \\
	&\quad -  \sum_{((x_1, j_1), \ldots, (x_n, j_n)) \in \Z_K^{(n)}} \nu\bigg( \bigcap_{k=1}^n \{\eta(x_k) \ge j_k\} \bigg) \P\bigg( \bigcap_{k=1}^n \{\xi_t^{x_kj_k} \in A\} \bigg). 
\end{aligned}
\end{equation}

We start by proving \eqref{eq:icmonotonicity}, namely that 
\begin{equation*}\label{eq:icmonotonicityagain}
	0 \le (\mu_t^\eta(A))^n - \E_\eta[N_t^{(n)}(A^n)] \le (\mu_t^{\chi}(A))^n - \E_{\chi}[N_t^{(n)}(A^n)], 
\end{equation*}
whenever $\eta, \chi \in \sX_K$ such that $\eta(x) \le \chi(x)$ for all $x \in \Z$.

\begin{proof}[\bf Proof of \eqref{eq:icmonotonicity}] From \eqref{eq:meanpwrfactmomdiff} applied to the determinisitc profile $\eta$, 
\begin{align*}
	&(\mu^\eta_t(A))^n - \E_\eta[N_t^{(n)}(A^n)] \\
	&=  \sum_{((x_1, j_1), \ldots, (x_n, j_n)) \in \Z_K^n} \prod_{k=1}^n 1(\eta(x_k) \ge j_k)P_{x_k}(\zeta_{Kt} \in A) \\
	&\quad -  \sum_{((x_1, j_1), \ldots, (x_n, j_n)) \in \Z_K^{(n)}} 1\bigg( \bigcap_{k=1}^n \{\eta(x_k) \ge j_k\} \bigg) \P\bigg( \bigcap_{k=1}^n \{\xi_t^{x_kj_k} \in A\} \bigg) \\
	&=  \sum_{((x_1, j_1), \ldots, (x_n, j_n)) \in \Z_K^{(n)}} \bigg( \prod_{k=1}^n 1(\eta(x_k) \ge j_k) \bigg) \bigg( \prod_{k=1}^n P_{x_k}(\zeta_{Kt} \in A) - \P\bigg( \bigcap_{k=1}^n \{\xi_t^{x_kj_k} \in A\} \bigg) \bigg) \\
	&\quad + \sum_{((x_1, j_1), \ldots, (x_n, j_n)) \in \Z_K^n \setminus \Z_K^{(n)}} \bigg( \prod_{k=1}^n 1(\eta(x_k) \ge j_k) \bigg) \prod_{k=1}^m P_{x_k}(\zeta_{Kt} \in A). 
\end{align*}
From Corollary \ref{negassoc}, 
\[
	\prod_{k=1}^n P_{x_k}(\zeta_{Kt} \in A) - \P\bigg( \bigcap_{k=1}^n \{\xi_t^{x_kj_k} \in A\} \bigg) \ge 0. 
\]
In particular, $(\mu^\eta_t(A))^n - \E_\eta[N_t^{(n)}(A^n)] \ge 0$. Moreover, since by assumption $\prod_{k=1}^n 1(\eta(x_k) \ge j_k) \le \prod_{k=1}^n 1(\chi(x_k) \ge j_k)$, this implies \eqref{eq:icmonotonicity}. 
\end{proof}

Below we complete the proof of Proposition \ref{genmzrbound}, which requires verifying \eqref{eq:icmonotonicity2}. Recall that $\nu \in \sP_K$ and $\chi_\nu = K1_{\{\nu > 0\}}$. 

\begin{proof}[\bf Proof of Proposition \ref{genmzrbound}]
Jensen's inequality gives 
\[
	(\mu_t^\nu(A))^n = (E_{\nu}[\mu^\eta_t(A)])^n \le E_{\nu}[(\mu_t^\eta(A))^n]. 
\]
By definition of $\chi_\nu$, $\nu(\eta(x) \le \chi_\nu(x)) = 1$ for all $x \in \Z$. So, we have from \eqref{eq:icmonotonicity} that 
\begin{align*}
	(\mu_t^\nu(A))^n - \E_\nu[N_t^{(n)}(A^n)] &\le E_\nu\big[ (\mu^\eta(A))^n - \E_\eta[N_t^{(n)}(A^n)] \big] \\
	&\le (\mu^{\chi_\nu}(A))^n - \E_{\chi_\nu}[N_t^{(n)}(A^n)], 
\end{align*}
which is the upper bound in \eqref{eq:icmonotonicity2}. 

Now we turn to the lower bound in \eqref{eq:icmonotonicity2}. 
When $(x_1, \ldots, x_n) \in \Omega_n^1 = \Z^{(n)}$ (i.e., the $x_i$ are distinct), then $((x_1,j_1), \ldots, (x_n,j_n)) \in \Z_K^{(n)}$ for any $j_1, \ldots, j_n \in [K]$. Let 
\[
	\widetilde \Z_K^{(n)} = \{((x_1, j_1), \ldots, (x_n, j_n)) \in \Z_K^{(n)} : (x_1, \ldots, x_n) \in \Omega_n^1\}
\]
denote this subset of $\Z_K^{(n)}$.  
Because $\nu$ is a product measure,
\[
	\nu \bigg( \bigcap_{k=1}^n \{\eta(x_k) \ge j_k\} \bigg) = \prod_{k=1}^n \nu(\eta(x_k) \ge j_k), 
\]
for any 
$((x_1,j_1), \ldots, (x_n, j_n)) \in \widetilde \Z_K^{(n)}$. Then from \eqref{eq:meanpwrfactmomdiff} we have 
\begin{align*}
	&(\mu^\nu_t(A))^n - \E_\nu[N_t^{(n)}(A^n)] \\
	&= \sum_{((x_1,j_1), \ldots, (x_n,j_n)) \in \widetilde \Z_K^{(n)}} \bigg( \prod_{k=1}^n \nu(\eta(x_k) \ge j_k) \bigg) \bigg( \prod_{k=1}^n P_{x_k}(\zeta_{Kt} \in A) - \P\bigg( \bigcap_{k=1}^n \{\xi_t^{x_kj_k} \in A\} \bigg) \bigg) \\
	&\quad + \sum_{((x_1,j_1), \ldots, (x_n,j_n)) \in \Z_K^{n} \setminus \widetilde \Z_K^{(n)}} \bigg( \prod_{k=1}^n \nu(\eta(x_k) \ge j_k) \bigg) \prod_{k=1}^n P_{x_k}(\zeta_{Kt} \in A) \\
	&\quad - \sum_{((x_1,j_1), \ldots, (x_n,j_n)) \in \Z_K^{(n)} \setminus \widetilde \Z_K^{(n)}} \nu\bigg( \bigcap_{k=1}^n \{\eta(x_k) \ge j_k\} \bigg) \P\bigg( \bigcap_{k=1}^n \{\xi_t^{x_kj_k} \in A\} \bigg). 
\end{align*}
By Corollary \ref{negassoc}, the first two terms on the right hand side of the above display are nonnegative, so that 
\begin{equation}\label{eq:dropterms}
	(\mu^\nu_t(A))^n - \E_\nu[N_t^{(n)}(A^n)]  \ge - \sum_{((x_1,j_1), \ldots, (x_n,j_n)) \in \Z_K^{(n)} \setminus \widetilde \Z_K^{(n)}} \nu\bigg( \bigcap_{k=1}^n \{\eta(x_k) \ge j_k\} \bigg) \P\bigg( \bigcap_{k=1}^n \{\xi_t^{x_kj_k} \in A\} \bigg). 
\end{equation}
Note that, by definition of $\widetilde \Z_K^{(n)}$, 
$((x_1,j_1), \ldots, (x_n,j_n)) \in \Z_K^{(n)} \setminus \widetilde \Z_K^{(n)}$ implies $(x_1, \ldots, x_n) \in (\Omega_n^1)^c$. Then by Corollary \ref{negassoc}, 
\begin{equation}\label{eq:intermsofU}
\begin{aligned}
	&\sum_{((x_1,j_1), \ldots, (x_n,j_n)) \in \Z_K^{(n)} \setminus \widetilde \Z_K^{(n)}} \nu\bigg( \bigcap_{k=1}^n \{\eta(x_k) \ge j_k\} \bigg) \P\bigg( \bigcap_{k=1}^n \{\xi_t^{x_kj_k} \in A\} \bigg) \\
	&\le \sum_{\bx \in (\Omega_n^1)^c} \sum_{\bj \in [K]^n} \nu\bigg( \bigcap_{k=1}^n \{\eta(x_k) \ge j_k\} \bigg) \prod_{k=1}^n P_{x_k}(\xi_{Kt} \in A) \le K^n \sum_{\bx \in (\Omega_n^1)^c \times \{\nu > 0\}^n} U_n^K(t) 1_{A^n}(\bx). 
\end{aligned}
\end{equation}
Moreover, by Lemma \ref{kappataubounds}, 
\begin{equation}\label{eq:applyUbound}
	K^n \sum_{\bx \in (\Omega_n^1)^c \times \{\nu > 0\}^n} U_n^K(t) 1_{A^n}(\bx) \le K^n {n \choose 2} \big( \il 1_{\{\nu > 0\}}, U_1^1(Kt)1_A \ir \big)^{n-2} \tau_{Kt}(\{\nu > 0\}, A). 
\end{equation}
As in \eqref{eq:iptomuinpf}, $K\il 1_{\{\nu > 0\}}, U_1^1(Kt)1_A \ir = \mu_t^{\chi_\nu}(A)$. Thus, combining \eqref{eq:dropterms}, \eqref{eq:intermsofU}, and \eqref{eq:applyUbound} we obtain 
\begin{align*}
	(\mu^\nu_t(A))^n - \E_\nu[N_t^{(n)}(A^n)] 
	&\ge - K^2 {n \choose 2} \big( \mu_t^{\chi_\nu}(A) \big)^{n-2} \tau_{Kt}(\{\nu > 0\}, A), 
\end{align*}
which is the lower bound in \eqref{eq:icmonotonicity2}.
\end{proof}

\subsection{Proof of Proposition \ref{kappatauconv}}

Part (a) of Proposition \ref{kappatauconv} will follow from the next two lemmas. Lemma \ref{kappato0}, shown in \cite{ConSet2023}, is an estimate involving a single random walk and follows from precise tail bounds and a local central limit theorem. 

\begin{lemma}[{\cite[Lemmas 6.1 and 6.2]{ConSet2023}}]\label{kappato0} Assume Condition \ref{lighttail}, and let $c > 0$ be an arbitrary constant. If $z_t \sim c\sqrt{t\log t}$ as $t \to \infty$, then 
\[
	\lim_{t\to\infty} \sum_{x \in \Z} \int_0^t P_x(\zeta_s = 0)^2P_x(\zeta_{t-s} \ge z_t)^2\,ds = 0. 
\]
Moreover, if $z_t \sim c \sqrt{t\log L}$ as $t \to \infty$, where $L \uparrow \infty$ and $L = o(\sqrt{t/\log t})$ as $t \to \infty$, then 
\[
	\lim_{t\to\infty} \sum_{x \in \Z} \int_0^t (P_x(\zeta_s = 0) - P_x(\zeta_s = -L))^2P_x(\zeta_{t-s} \ge z_t)^2\,ds = 0. 
\]
\end{lemma}

For the following lemma, recall that all moments of $p$ exist under Condition \ref{lighttail}. 

\begin{lemma}\label{kappabound2} Assume Condition \ref{lighttail}. For any $b > 0$ and $L \in \N \cup \{\infty\}$, 
\begin{align*}
	\kappa_t((b, \infty), (-L, 0]) \le \sigma^2 \sum_{x \in \Z} \int_0^t (P_{x}(\zeta_s = 0) - P_{x}(\zeta_s = -L))^2 P_{x}(\zeta_{t-s} > b/2)^2\,ds +  \frac{8M_4t}{b^2},  
\end{align*}
where $M_4 = \sum_{y \in \Z} y^4 p(0,y)$, and with the convention that $P_x(\zeta_s = -\infty) = 0$. 
\end{lemma}

\begin{proof} 
At several points in the proof, we use the fact that translation invariance and symmetry of $p$ imply
\begin{equation}\label{eq:RWMP}
	P_x(\zeta_t > y) = P_0(\zeta_t > y-x) = P_0(\zeta_t > x - y) = P_y(\zeta_t > x), 
\end{equation}
for any $x, y \in \Z$ and $t \ge 0$. 

Now, translation invariance of $p$ implies 
\begin{align*}
	&\kappa_t((b, \infty), (-L, 0]) \\
	&= \sum_{x, y \in \Z} p(x,y) \int_0^t (P_x(\zeta_s \in (-L, 0]) - P_y(\zeta_s \in (-L, 0]))^2P_x(\zeta_{t-s} > b)P_y(\zeta_{t-s} > b)\,ds \\
	&= 2 \sum_{x\in \Z} \sum_{y \in \N} p(0, y) \int_0^t (P_{x}(\zeta_s \in (-L, 0]) - P_{x-y}(\zeta_s \in (-L, 0]))^2 P_{x}(\zeta_{t-s} > b)P_{x-y}(\zeta_{t-s} > b)\,ds. 
\end{align*}
Since \eqref{eq:RWMP} implies $P_{x-y}(-L < \zeta_s \le 0) = P_x(-L + y < \zeta_s \le y)$, for $y \ge 0$ we have 
\begin{align*}
	&(P_{x}(\zeta_s \in (-L, 0]) - P_{x-y}(\zeta_s \in (-L, 0]))^2 = (P_x(-L < \zeta_s \le -L + y) - P_x(0 < \zeta_s \le y))^2 \\
	&= \bigg(\sum_{0 < k \le y} (P_x(\zeta_s = k) - P_x(\zeta_s = -L + k))\bigg)^2 \le y\sum_{0 < k \le y} (P_x(\zeta_s = k) - P_x(\zeta_s = -L + k))^2. 
\end{align*}
Moreover, \eqref{eq:RWMP} implies that $P_x(\zeta_s = k) - P_x(\zeta_s = -L + k) = P_{x-k}(\zeta_s = 0) - P_{x-k}(\zeta_s = -L)$, and for $y \ge 0$ \eqref{eq:RWMP} implies $P_{x-y}(\zeta_{t-s} > b) \le P_x(\zeta_{t-s} > b)$. Therefore, changing variables in the first equality and appealing again to \eqref{eq:RWMP} in the second equality, 
\begin{align*}
	&\kappa_t((b, \infty), (-L, 0]) \\
	&\le 2  \sum_{y \in \N} yp(0,y) \sum_{0 < k \le y} \sum_{x \in \Z} \int_0^t (P_{k-x}(\zeta_s = 0) - P_{x-k}(\zeta_s = -L))^2 P_x(\zeta_{t-s} > b)^2\,ds \\
	&= 2  \sum_{y \in \N} yp(0,y) \sum_{0 < k \le y} \sum_{x \in \Z} \int_0^t (P_{x}(\zeta_s = 0) - P_{x}(\zeta_s = -L))^2 P_{x+k}(\zeta_{t-s} > b)^2\,ds \\
	&= 2  \sum_{y \in \N} yp(0,y) \sum_{0 < k \le y} \sum_{x \in \Z} \int_0^t (P_{x}(\zeta_s = 0) - P_{x}(\zeta_s = -L))^2 P_{x}(\zeta_{t-s} > b-k)^2\,ds \\
	&\le 2  \sum_{y \in \N} y^2p(0,y) \sum_{x \in \Z} \int_0^t (P_{x}(\zeta_s = 0) - P_{x}(\zeta_s = -L))^2 P_{x}(\zeta_{t-s} > b-y)^2\,ds. 
\end{align*}

Now, set 
\[
	\sI_t(L, u) = \sum_{x \in \Z} \int_0^t (P_{x}(\zeta_s = 0) - P_{x}(\zeta_s = -L))^2 P_{x}(\zeta_{t-s} > u)^2\,ds, 
\]
and observe that this quantity is decreasing in $u$, and also that, using \eqref{eq:RWMP}, 
\begin{align*}
	\sup_{u \in \R} \sI_t(L, u) &\le \int_0^t \sum_{x \in \Z} (P_x(\zeta_s = 0) + P_x(\zeta_s = -L))\,ds \\
	&= \int_0^t \sum_{x \in \Z} (P_0(\zeta_s = x) + P_0(\zeta_s = x + L))\,ds = 2t. 
\end{align*}

Then, 
\begin{align*}
	\kappa_t((b, \infty), (-L, 0]) &\le 2 \sum_{y \in \N} y^2 p(0,y) \sI_t(L, b - y) \\
	&\le 2\sum_{0 \le y \le b/2} y^2 p(0,y) \sI_t(L, b - y) + 4t \sum_{y > b/2} y^2 p(0,y) \\
	&\le \sigma^2 \sI_t(L, b/2) + \frac{8t}{b^2} \sum_{y \in \Z} y^4 p(0,y),  
\end{align*}
where in the last line, symmetry of $p$ was used to drop a factor of $2$ and write the bound in terms of $\sigma^2$. 
\end{proof}

Now we prove Proposition \ref{kappatauconv} (a). 

\begin{proof}[\bf Proof of Proposition \ref{kappatauconv} (a)]
Any $A \in \sU$ is bounded, and so $A \subset (a, \infty)$ for $a > - \infty$. Note that $\kappa_t$ is monotone in its first component:
\begin{equation}\label{eq:kappamonotone}
	\kappa_t(v_t^{-1}(A), (-\infty, 0]) \le \kappa_t(v_t^{-1}(a, \infty), (-\infty, 0]) = \kappa_t((v_t^{-1}(a), \infty), (-\infty, 0]). 
\end{equation}
From \eqref{eq:zfirstorder}, $v_t^{-1}(a)>0$ for sufficiently large $t$. Applying Lemma \ref{kappabound2} to the right hand side of \eqref{eq:kappamonotone}, 
\begin{equation}\label{eq:kappaIbound}
	\kappa_t(v_t^{-1}(A), (-\infty, 0]) \le \sigma^2 \sum_{x \in \Z} \int_0^t P_x(\zeta_s = 0)^2P_x(\zeta_{t-s} > v_t^{-1}(a)/2)^2\,ds + \frac{8M_4 t}{(v_t^{-1}(a))^2}. 
\end{equation}
Similarly, 
\begin{equation}\label{eq:kappamonotoneL}
\begin{aligned}
	&\kappa_t(v_{t,L}^{-1}(A), (-L, 0]) \\
	&\le \sigma^2 \sum_{x \in \Z} \int_0^t (P_{x}(\zeta_s = 0) - P_{x}(\zeta_s = -L))^2 P_{x}(\zeta_{t-s} > v_{t,L}^{-1}(a)/2)^2\,ds +  \frac{8M_4t}{(v_{t,L}^{-1}(a))^2}. 
\end{aligned}
\end{equation}
Since $v_t^{-1}(a)/2 \sim (\sigma/2)\sqrt{t\log t}$ and $v_{t,L}^{-1}(a)/2 \sim (\sigma / 2)\sqrt{t\log L}$ as $t \to \infty$ (cf. \eqref{eq:zfirstorder} and \eqref{eq:zLfirstorder}), the result follows from \eqref{eq:kappaIbound}, \eqref{eq:kappamonotoneL}, and Lemma \ref{kappato0}. 
\end{proof}

For part (b) of Proposition \ref{kappatauconv}, we first have the following lemma. Recall the notation $\{\eta > 0\} = \{x : \eta(x) > 0\}$ for $\eta \in \sX_K$. 

\begin{lemma}\label{tauto0} If $\eta \in \{0, K\}^\Z \cap \sX_K^{\text{step}}$ 
and $A \in \sB_\R$ with $a = \inf A > - \infty$, then
\[
	\max\big\{ \tau_t(A, \{\eta > 0\}) , \tau_t(\{\eta > 0\}, A) \big\} \le K^{-1} \mu_{t/K}^\eta(A) P_0(\zeta_t > a). 
\]
\end{lemma}

\begin{proof} Since $\tau_t$ is monotone in both its entries, it suffices to assume that $A = (a, \infty)$. 
Note that 
\[
	K^{-1}\mu_{t/K}^\eta(A) = \sum_{x > a} P_x(\zeta_t \in \{\eta > 0\}) = \sum_{x \in \{\eta > 0\}} P_x(\zeta_t > a). 
\]
Moreover, when $x > a$, since $\eta(y) = 0$ for $y > 0$, 
\[
	P_x(\zeta_t \in \{\eta > 0\}) \le P_x(\zeta_t \le 0) = P_0(\zeta_t \le -x) = P_0(\zeta_t \ge x) \le P_0(\zeta_t > a), 
\]
and for $x \in \{\eta > 0\}$, 
\[
	P_x( \zeta_t > a) = P_0(\zeta_t > a - x) \le P_0(\zeta_t > a). 
\]
Then we have 
\begin{align*}
	&\max\big\{ \tau_t(A, \{\eta > 0\}) , \tau_t(\{ \eta > 0 \}, A) \big\} \\
	&= \max\bigg\{ \sum_{x \in A \cap \Z} P_x(\zeta_t \in \{\eta > 0\})^2, \sum_{x \in \{\eta > 0\}} P_x(\zeta_t \in A)^2 \bigg\} \\
	&\le P_0(\zeta_t > a) \max\bigg\{ \sum_{x \in A \cap \Z} P_x(\zeta_t \in \{\eta > 0\}), \sum_{x \in \{\eta > 0\}} P_x(\zeta_t \in A) \bigg\} \\
	&= P_0(\zeta_t > a)K^{-1}\mu_{t/K}^\eta(A). \qedhere
\end{align*}
\end{proof}

For the proof of Proposition \ref{kappatauconv} (b), fix $A \in \sU$, and let $a = \inf A$. Since $A$ is bounded, $a > - \infty$. 

\begin{proof}[\bf Proof of Proposition \ref{kappatauconv} (b)] 
First suppose that $L(\frac{\log t}{t})^{1/2} \to \psi \in (0, \infty]$ as $t \to \infty$. By \eqref{eq:zfirstorder} and the central limit theorem, $\lim_{t\to\infty} P_0(\zeta_t > v_t^{-1}(a)) = 0$. 
 By Proposition \ref{meanconv} (a), $\sup_t \mu^L_{t/K}\circ v_t^{-1}(A) < \infty$. Because $v_t$ is nondecreasing, it follows from Lemma \ref{tauto0} that 
\begin{align*}
	\max\big\{ \tau_t(v_t^{-1}(A), (-L,0]) , \tau_t((-L,0], v_t^{-1}(A)) \big\} &\le K^{-1} \mu^L_{t/K} \circ v_t^{-1}(A) P_0(\zeta_t > v_t^{-1}(a)) 
	\underset{t\to\infty}{\longrightarrow} 0. 
\end{align*}

Next suppose $L \uparrow \infty$ such that $L(\frac{\log t}{t})^{1/2} \to 0$ as $t \to \infty$. By \eqref{eq:zLfirstorder} and the central limit theorem, $\lim_{t\to\infty} P_0(\zeta_t > v_{t,L}^{-1}(a)) = 0$, and by Proposition \ref{meanconv} (b), 
$\sup_t \mu^L_{t/K} \circ v_{t,L}^{-1}(A) < \infty$. Because $v_{t,L}$ is nondecreasing, it follows from Lemma \ref{tauto0} that 
\begin{align*}
	\max\big\{ \tau_t(v_{t,L}^{-1}(A), (-L,0]) , \tau_t((-L,0], v_{t,L}^{-1}(A)) \big\} &\le K^{-1} \mu^L_{t/K} \circ v_{t,L}^{-1}(A) P_0(\zeta_t > v_{t,L}^{-1}(a)) \\
	&\underset{t\to\infty}{\longrightarrow} 0. \qedhere
\end{align*}
\end{proof}

\end{document}